\documentclass[12pt]{amsart}
\usepackage{amsmath}
\usepackage{amsthm}
\usepackage{amssymb}
\usepackage{amscd}
\usepackage{amsfonts}
\usepackage{amsbsy}
\usepackage{epsfig,afterpage}
\usepackage[dvips]{psfrag}
\usepackage{subfigure}
\usepackage{epsfig}
\usepackage{amsmath}
\usepackage{srcltx}
\usepackage{amsfonts}
\usepackage{amssymb}
\usepackage{psfrag}
\usepackage{verbatim}
\usepackage{graphicx}
\usepackage{amsmath}
\usepackage{amsthm}
\usepackage{amssymb}
\usepackage{amscd}
\usepackage{amsfonts}
\usepackage{amsbsy}
\usepackage{epsfig}
\usepackage[dvips]{psfrag}
\usepackage{multirow}
\usepackage{graphics}
\usepackage{epstopdf}
\usepackage{overpic}
\usepackage{float}

\usepackage{graphicx}

\usepackage[colorlinks=true, linkcolor=blue, citecolor=red]{hyperref}

\newtheorem {theorem} {Theorem}
\newtheorem {proposition} [theorem]{Proposition}
\newtheorem {corollary} [theorem]{Corollary}
\newtheorem {lemma}  [theorem]{Lemma}
\newtheorem {example} [theorem]{Example}
\newtheorem {remark} [theorem]{Remark}

\newtheorem {definition} [theorem]{Definition}

\newtheorem{mtheorem}{Theorem}

\usepackage{float}
\usepackage{tikz, wrapfig}
\tikzset{node distance=3cm, auto}
\allowdisplaybreaks

\begin{document}

\title[SF normal forms and PSVF's]
{Slow-fast normal forms arising from piecewise smooth vector fields}

\author[O. H. Perez, G. Rondón and P. R. da Silva]
{Otavio Henrique Perez$^{1,2}$, Gabriel Rondón$^{1}$ and Paulo Ricardo da Silva$^{1}$}

\address{$^{1}$S\~{a}o Paulo State University (Unesp), Institute of Biosciences, Humanities and
	Exact Sciences. Rua C. Colombo, 2265, CEP 15054--000. S. J. Rio Preto, S\~ao Paulo,
	Brazil.}
\address{$^{2}$Instituto Federal de São Paulo (IFSP), Rua Doutor Aldo Benedito, R. Sebastião Pierri, 250, CEP 14804--296. Araraquara, São Paulo, Brazil.}

\email{oh.perez@unesp.br}
\email{gabriel.rondon@unesp.br}
\email{paulo.r.silva@unesp.br}

\thanks{ .}


\date{}
\dedicatory{Dedicated to the memory of Jorge Sotomayor Tello.}
\maketitle

\begin{abstract}

We studied piecewise smooth differential systems of the form
$$\dot{z} = Z(z) = \dfrac{1 + \operatorname{sgn}(F)}{2}X(z) + \dfrac{1 - \operatorname{sgn}(F)}{2}Y(z),$$ where $F: \mathbb{R}^{n}\rightarrow \mathbb{R}$ is a smooth map having 0 as a regular value. We consider linear regularizations of the vector field $Z$ 
given by
$$\dot{z}= Z_{\varepsilon}(z) = \dfrac{1 + \varphi(F/\varepsilon)}{2}X(z) +\dfrac{1 - \varphi(F /\varepsilon)}{2}Y(z),$$where $\varphi$ is a transition function (not necessarily monotonic) 
and nonlinear regularizations of the vector field $Z$ whose  transition function is monotonic. It is a well-known fact that the regularized system is a slow–fast system. The main contribution of this paper is the study of typical singularities of slow-fast systems that arise from (linear or nonlinear) regularizations. We developed an algorithm to construct suitable transition functions, and we apply these ideas in order to create slow-fast singularities from normal forms of piecewise smooth vector fields. We present examples of transition functions that, after regularization of a PSVF normal form, generate normally hyperbolic, fold, transcritical, and pitchfork singularities.


\end{abstract}


\section{Introduction}
In real life there are phenomena whose mathematical models are expressed by piecewise smooth vector fields, which have been studied at least since 1937. These systems are used in many branches of applied sciences, for example, Physics, Control Theory, Economics, Cell Mitosis, etc. For more details see, for instance, \cite{bernardo,Filippov}.

A piecewise smooth vector field (or PSVF for short) is defined as follows: let $\Sigma$ be a subset of the ambient space (for example, a manifold embedded in $\mathbb{R}^{n}$). Such subset is called \emph{discontinuity locus} and it divides the ambient space in finitely many open subsets $\{U_{i}\}_{i = 1}^{k}$. In each open subset $U_{i}$ is defined a smooth vector field. This paper deals with the case where a smooth curve divides a neighbourhood of $0\in\mathbb{R}^{2}$ in two open regions. See Section \ref{sec-pwsvf} for a precise definition.

One of the most important question concerning PSVF's is: how to define the dynamics in $\Sigma$? In other words, how to define the transition between the dynamics defined in two different open sets?

Filippov \cite{Filippov} gave an answer defining the dynamics in $\Sigma$ as the convex combination of two vector fields. This defines the so called \emph{Sliding vector field}. We say that this vector field defined according to Filippov's ideas follows the \emph{Filippov's convention}.

However, for some models, the Filippov's convention is not sufficient to describe the dynamics. For example, in \cite{Jeffrey} a model involving friction between an object and a flat surface was studied. The author gave an example that Filippov's convention takes into account only kinetic friction, while it is possible to consider static friction as well.

Another way to define the dynamics in the discontinuity locus $\Sigma$ is combining two powerful tools: Regularizations of PSVF's and Blow-ups. A regularization process that is compatible with the Filippov's convention is the \emph{Sotomayor-Teixeira regularization} \cite{SotoTeixeira}, which consists in obtaining a one-parameter family of smooth vector fields $Z_{\varepsilon}$ converging to $Z$ when $\varepsilon\to 0$ (see Subsection \ref{subsec-reg}). By using blow-up techniques, the regularized system $\dot{z} = Z_{\varepsilon}(z)$ becomes a slow-fast system, and therefore we are able to apply classical results on geometric singular perturbation theory (see Subsection \ref{subsec-gspt}) in the study of PSVF's. Such a link between Regularization Processes and geometric singular perturbation theory is a recent approach in mathematics and we refer to \cite{BuzziSilvaTeixeira, LlibreSilvaTeixeira,LlibreSilvaTeixeira2,NovaesJeffrey,PanazzoloSilva, SilvaSarmientoNovaes} for further details. A similar approach can also be seen in \cite{KristiansenHogan}.

Different regularization processes lead to different slow-fast systems, which gives rise to different sliding, escaping or sewing regions (see \cite{PanazzoloSilva,SilvaSarmientoNovaes,SilvaMoraes}). In this paper, we consider \emph{linear} regularizations and \emph{nonlinear} regularizations. See subsections \ref{subsec-reg} and \ref{subsec-non-linear-regularizations} for precise definitions.

The dynamics of the \emph{linearly} regularized system depends on the so called \emph{transition function} $\varphi$, which can be monotonic or non monotonic. These results are well known, and in this paper we recall them highlighting the relation between properties of the graphic of $\varphi$, properties of the slow-fast system and sliding regions of PSVF's. See Theorem \ref{mtheorem-A} below.

The main goal of this paper is to study typical singularities of slow-fast systems that arise from (linear or nonlinear) regularizations. Concerning linear regularizations, we developed an algorithm to construct suitable transition functions (see Appendix \ref{sec-trans-func}), and we apply these ideas in order to create slow-fast singularities from normal forms of piecewise smooth vector fields. For both linear and nonlinear regularizations are presented examples of PSVF's such that, after (linear or nonlinear) regularization and directional blow-up, the slow-fast system presents normally hyperbolic, fold, transcritical or pitchfork singularities.

At some point, the reader may think that, after \emph{linear} regularization and blow-up, it is possible to generate any slow-fast singularity, since it is just a matter of a suitable choice of the transition function. In general, this is not true. Indeed, we show that it does not exist a transition function that generate a pitchfork singularity. However, if we consider \emph{nonlinear regularizations} it is possible to generate such a singularity (see Example \ref{exe-non-linear-reg}). This shows that nonlinear regularizations are more general than the linear ones (see also \cite{NovaesJeffrey,SilvaSarmientoNovaes}).

Our main results, Theorems \ref{mtheorem-A}, \ref{teob} and \ref{teoc} are stated and proved in Section \ref{mainresults}. In what follows, we briefly describe them.



Firstly, consider linear regularizations. Suppose that we drop the monotonicity condition of the transition function $\varphi$. In this context, we will prove that the critical points of $\varphi$ give rise to non normally hyperbolic points of the critical set $C_{0}$ of $\dot{z}=Z_\varepsilon(z)$. For more details see Item (a) of Theorem \ref{mtheorem-A}. 

In addition, item (b) of Theorem \ref{mtheorem-A} assures that we extend the classical Filippov sliding region when the transition function satisfy $|\varphi(x_{0})| > 1$ for some $x_{0}$ in the open interval $(-1,1)$. According to item (c) of the same Theorem, the dynamics in this extended sliding region is naturally defined using the Filippov sliding vector field.

Finally, item (d) of Theorem \ref{mtheorem-A} says that there are cases in which it is not possible to apply geometric singular perturbation theory in order to define the sliding dynamics in some points of $\Sigma$. See Figure \ref{fig-intro-trans-func}.

\begin{figure}[h]
  \center{\includegraphics[width=1\textwidth]{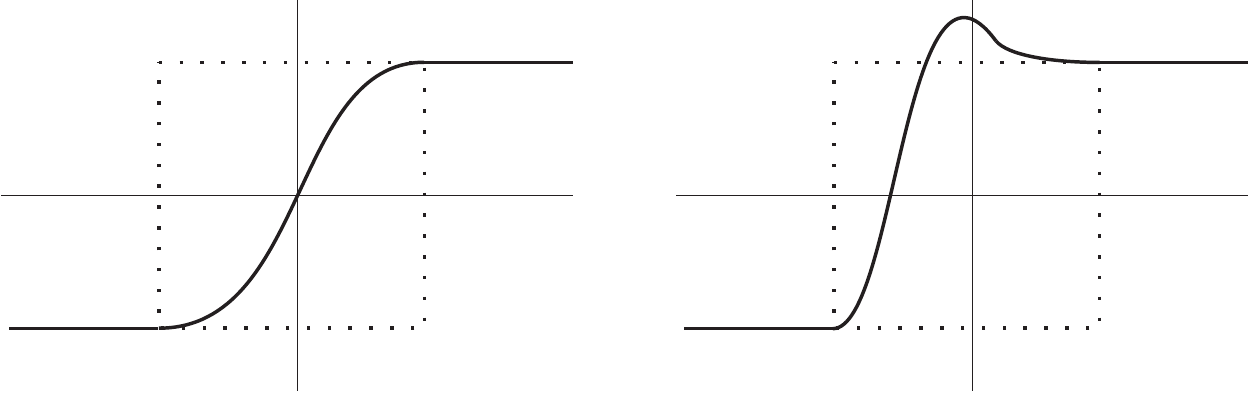}}\\
  \caption{\footnotesize{Monotonic transition function (left) and non monotonic transition function (right). The monotonic one generates only normally hyperbolic critical sets, and the sliding region coincides with the one proposed by Filippov. The non monotonic one has a critical point, which generates a non normally hyperbolic point of the critical manifold. Moreover, in this example, such a transition function extends the classical notion of sliding region.}}
  \label{fig-intro-trans-func}
\end{figure}

Slow-fast normal forms are well known in the literature (see Subsection \ref{sub-sec-normal-forms-sf} and the references therein). In Theorem \ref{teob} we state conditions that both PSVF and transition function must satisfy in order to generate classical slow-fast normal forms, such as fold and transcritical singularities. We apply Lemma \ref{lemma-trans-function-tech} in order to construct suitable transition functions for each slow-fast singularity (See Appendix \ref{sec-trans-func}). Moreover, we prove that there are slow-fast normal forms that can not be generated by \emph{linear} regularization processes. This is the case of the pitchfork singularity. See Figure \ref{fig-intro-sf-sing}.

In order to generate pitchfork singularities, we must consider nonlinear regularization. Theorem \ref{teoc} gives the conditions that must satisfy both monotonic transition function and vector field associated with the nonlinearly regularized system to generate this type of singularity.

\begin{figure}[h]
  \center{\includegraphics[width=1\textwidth]{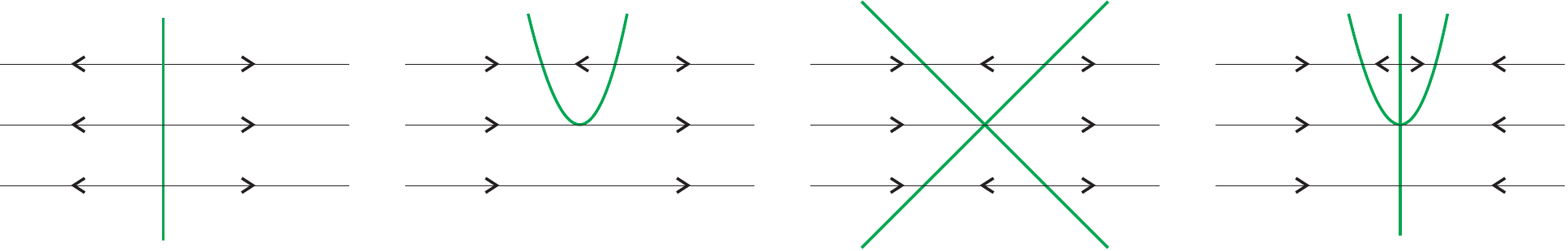}}\\
  \caption{\footnotesize{From the left to the right: normally hyperbolic, fold, transcritical and pitchfork points of a slow-fast system. It is not possible to generate the last one with linear regularizations, for any transition function. However, it is possible to generate it with nonlinear regularizations. The critical set is highlighted in green.}}
  \label{fig-intro-sf-sing}
\end{figure}

The paper is organized as follows. In Section \ref{sec-pwsvf} we present some introductory notions on PSVF, regularization processes, geometric singular perturbation theory and slow-fast normal forms. In section \ref{mainresults} we state and prove Theorems \ref{mtheorem-A}, \ref{teob}, and \ref{teoc}. Subsection \ref{subsec-fold-trans-pitch-reg} is dedicated to describe the dynamics of the (linearly and nonlinearly) regularized systems for $\varepsilon > 0$. In section \ref{sec:normalforms} we study normal forms of piecewise smooth vector fields and we investigate which slow-fast normal form can be generated from each PS-normal form. Structurally stable PSVF's and some codimention-1 bifurcations are considered. Finally, in Appendix \ref{sec-trans-func} we show how to build suitable transition functions that are used in our examples.



\section{Preliminaries on piecewise smooth vector fields and geometric singular perturbation theory}\label{sec-pwsvf}

\subsection{Piecewise smooth vector fields}\label{subsec-pwsvf}

\noindent

Let $F:U\subset\mathbb{R}^{n}\rightarrow\mathbb{R}$ be a sufficientlly smooth function and consider $C^{r}$ vector fields $X,Y:U\subset\mathbb{R}^{n}\rightarrow\mathbb{R}^{n}$. A \textit{$C^{r}$ piecewise smooth vector field} $Z:U\subset\mathbb{R}^{n}\rightarrow\mathbb{R}^{n}$ (or PSVF for short) is given by
\begin{equation}\label{piecewise-smooth}
Z(\mathbf{x}) = \displaystyle\frac{1}{2}\Bigg{(}\Big{(}1 + \operatorname{sgn}\big{(}F(\mathbf{x})\big{)}\Big{)}X(\mathbf{x}) +\Big{(}1 - \operatorname{sgn}\big{(}F(\mathbf{x})\big{)}\Big{)}Y(\mathbf{x}) \Bigg{)}
\end{equation}
where $\mathbf{x}\in U$ and we assume that $Z$ is multi-valued in the set
$$\Sigma = \{\mathbf{x}\in U; F(\mathbf{x}) = 0\},$$
which is called \textit{discontinuity locus} or \textit{discontinuity set}. The set of all $C^{r}$ piecewise smooth vector fields is denoted by $\Omega^{r}$. A PSVF is also denoted by $Z = (X,Y)$ in order to emphasize the dependency on the smooth vector fields $X$ and $Y$.

The \textit{Lie derivative} of $F$ with respect to the vector field $X$ is given by $XF = \langle X,\nabla F \rangle$ and $X^{i}F = \langle X,\nabla X^{i-1}F \rangle$ for all integer $i\geq 2$. This allows us to define the following regions in $\Sigma$:
\begin{enumerate}
  \item \textit{Filippov sewing region}: $$\Sigma^{w} = \big{\{}\mathbf{x}\in\Sigma \ | \ XF(\mathbf{x})\cdot YF(\mathbf{x})  > 0\big{\}};$$
  \item \textit{Filippov sliding region}: $$\Sigma^{s} = \big{\{}\mathbf{x}\in\Sigma\ | \  XF(\mathbf{x})  < 0,  YF(\mathbf{x})  > 0\big{\}};$$
  \item \textit{Filippov escaping region}: $$\Sigma^{e} =\big{\{}\mathbf{x}\in\Sigma \ | \  XF(\mathbf{x})  > 0,  YF(\mathbf{x})  < 0\big{\}}.$$
\end{enumerate}

We emphasize that in the literature these sets are simply called sewing region, sliding region and escaping region, respectively. Nevertheless, in \cite{PanazzoloSilva} the authors presented a new definition of such regions, which depends on the type of regularization adopted (see Definitions \ref{def-point-sliding} and \ref{def-point-sewing}). Due to this fact, we will call these regions as \textit{Filippov regions} in order to stress that we are talking about the classical definition of sewing, sliding and escaping. See Figure \ref{fig-def-filippov}.


A point $\mathbf{x}_{0}\in\Sigma$ is a \textit{PS-tangency point} if $XF(\mathbf{x}_{0}) = 0$ or $YF(\mathbf{x}_{0}) = 0$. We say that $\mathbf{x}_{0}$ is a \textit{PS-fold point} of $X$ if $XF(\mathbf{x}_{0}) = 0$ and $X^{2}F(\mathbf{x}_{0})\neq 0$. If $X^{2}F(\mathbf{x}_{0}) > 0$, $\mathbf{x}_{0}$ is a \textit{PS-visible fold} of $X$ and if $X^{2}F(\mathbf{x}_{0}) < 0$ we say that $\mathbf{x}_{0}$ is an \textit{PS-invisible fold} of $X$. Analogously we define PS-tangency points and PS-fold points of $Y$. Note that if $Y^{2}F(\mathbf{x}_{0}) < 0$, $\mathbf{x}_{0}$ is a PS-visible fold of $Y$ and if $Y^{2}F(\mathbf{x}_{0}) > 0$ the point $\mathbf{x}_{0}$ is and PS-invisible fold of $Y$. If $\mathbf{x}_{0}$ is a PS-fold of both $X$ and $Y$, we say that $\mathbf{x}_{0}$ is a \textit{PS-fold-fold}. Finally, we say that $\mathbf{x}_{0}\in\Sigma$ is a \textit{PS-cusp point} if $XF(\mathbf{x}_{0}) = X^{2}F(\mathbf{x}_{0}) = 0$  and $X^{3}F(\mathbf{x}_{0}) \neq 0$.

Singularities of slow-fast systems will be discussed later. Throughout this paper, a singularity of a PSVF will be called \textit{PS-singularity}, and a singularity of a slow-fast system when $\varepsilon = 0$ will be called \textit{SF-singularity}.

Following Filippov's convention \cite{Filippov}, one can define a vector field in $\Sigma^{s}\cup\Sigma^{e}\subset\Sigma$. The \textit{Filippov sliding vector field} associated to $Z\in\Omega^{r}$ is the vector field $Z^{\Sigma}:\Sigma\rightarrow\Sigma$ given by
\begin{equation}\label{sliding-vector-field}
Z^{\Sigma}(\mathbf{x}) = \displaystyle\frac{1}{YF - XF}\Big{(}X\cdot YF - Y\cdot XF\Big{)},
\end{equation}
which is the convex combination between $X$ and $Y$.

\subsection{Linear regularization of piecewise smooth vector fields}\label{subsec-reg}
\noindent

The regularization process proposed by Sotomayor and Teixeira in \cite{SotoTeixeira} is a powerfull tool in the study of piecewise smooth vector fields. With this technique, it is possible to construct a family of smooth vector fields $\{Z_{\varepsilon}\}_{\varepsilon}$ such that $Z_{\varepsilon}\rightarrow Z_{0} = Z$ when $\varepsilon\rightarrow 0$.

We say that $\varphi:\mathbb{R}\rightarrow\mathbb{R}$ is a \textit{transition function} if the following conditions are satisfied:
\begin{enumerate}
  \item $\varphi$ is sufficiently smooth;
  \item $\varphi(t) = -1$ if $t \leq -1$ and $\varphi(t) = 1$ if $t \geq 1$;
  \item $\varphi'(t) > 0$ if $s\in (-1,1)$. This condition is called \textit{monotonicity}.
\end{enumerate}

Throughout this paper it will be clear that, by dropping the monotonicity condition, it is possible to obtain different critical manifolds of the slow-fast system associated to the regularization. Moreover, non monotonic transition functions can expand the Filippov sliding region in $\Sigma$ (see \cite{PanazzoloSilva} and Theorem \ref{mtheorem-A} below).

\begin{definition}
Let $\varphi$ be a transition function. A linear regularization of a piecewise smooth vector field $Z = (X,Y)$ is an one-parameter family $Z_{\varepsilon}$ of smooth vector fields given by
\begin{equation}\label{reg-vector-field}
Z_{\varepsilon}(\mathbf{x}) = \Bigg{(}\displaystyle\frac{1}{2} + \displaystyle\frac{\varphi_{\varepsilon}\big{(}F(\mathbf{x})\big{)}}{2}\Bigg{)}X(\mathbf{x}) + \Bigg{(}\displaystyle\frac{1}{2} - \displaystyle\frac{\varphi_{\varepsilon}\big{(}F(\mathbf{x})\big{)}}{2}\Bigg{)}Y(\mathbf{x});
\end{equation}
with $\varphi_{\varepsilon}(s) = \varphi\Big{(}\displaystyle\frac{s}{\varepsilon}\Big{)}$ for $\varepsilon > 0$. When $\varphi$ is monotonic, we say that \eqref{reg-vector-field} is the ST-regularization (Sotomayor--Teixeira Regularization) of $Z$.
\end{definition}

Intuitively, regularizing piecewise smooth vector field means to replace the discontinuity set $\Sigma$ by a stripe (a tubular neighbourhood of $\Sigma$) of width $2\varepsilon$. Outside this stripe, the vector fields $Z_{\varepsilon}$ and $Z$ coincide, and inside the stripe the vector field $Z_{\varepsilon}$ can be seen as the ``average'' between $X$ and $Y$.

\begin{figure}[h]
  \center{\includegraphics[width=0.8\textwidth]{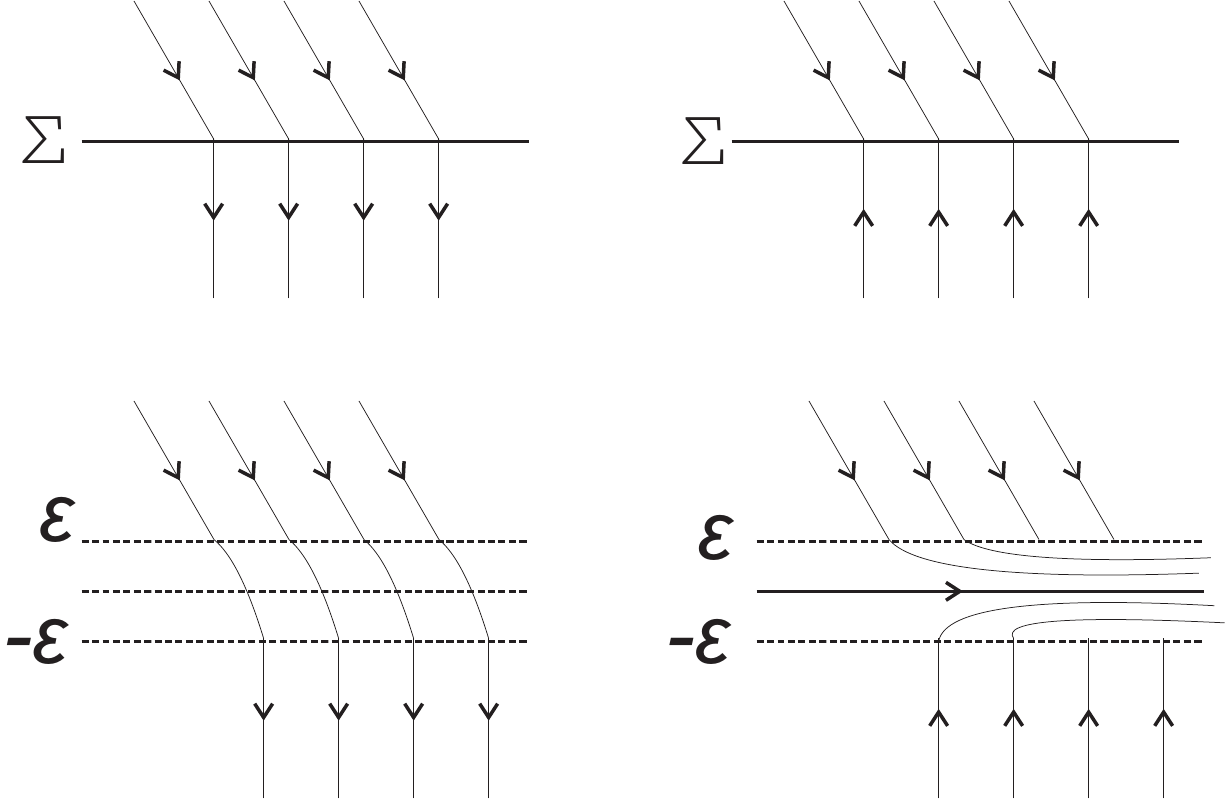}}\\
  \caption{\footnotesize{Regularization of a Filippov sewing region (left) and a Filippov sliding region (right).}}
   \label{fig-def-filippov}
\end{figure}

\subsection{Nonlinear regularization of piecewise smooth vector fields}\label{subsec-non-linear-regularizations}
In \cite{NovaesJeffrey,SilvaSarmientoNovaes} the authors considered another way to generalize the notions of sliding region and sliding vector field by means of \emph{nonlinear regularizations}.

\begin{definition}\label{def:nolinear}
A regularization $Z_{\varepsilon}$ between $X$ and $Y$ is called nonlinear if there exists a 1-parameter family of smooth vector fields
$\widetilde{Z}(\lambda,.)$, with $\lambda\in [-1, 1]$, such that
$\widetilde{Z}(-1, p) = Y(p)$, $\widetilde{Z}(1, p) = X(p)$ and $Z_{\varepsilon}(p) \in \{\widetilde{Z}(\lambda, p), \lambda\in [-1, 1]\}$, $\forall p \in U$.
\end{definition}

Analogously, we define the $\varphi-$nonlinear regularization of $ X$ and $Y$.

\begin{definition}\label{def:phinolinear}
A $\varphi-$nonlinear regularization of $X$ and $Y$ is the 1-parameter family given by $Z_{\varepsilon}(p) = \widetilde{Z}(\varphi(\frac{F}{\varepsilon}), p)$.
\end{definition} 
Recall that if $F > \varepsilon$, then $\varphi(\frac{F}{\varepsilon}) = 1$ and $Z_{\varepsilon} = X$; and if $F < -\varepsilon$, then $\varphi(\frac{F}{\varepsilon}) = -1$ and $Z_{\varepsilon} = Y$ (see Figure \ref{fig-nolinear}). 

\begin{figure}[h!]
  \center{\includegraphics[width=0.35\textwidth]{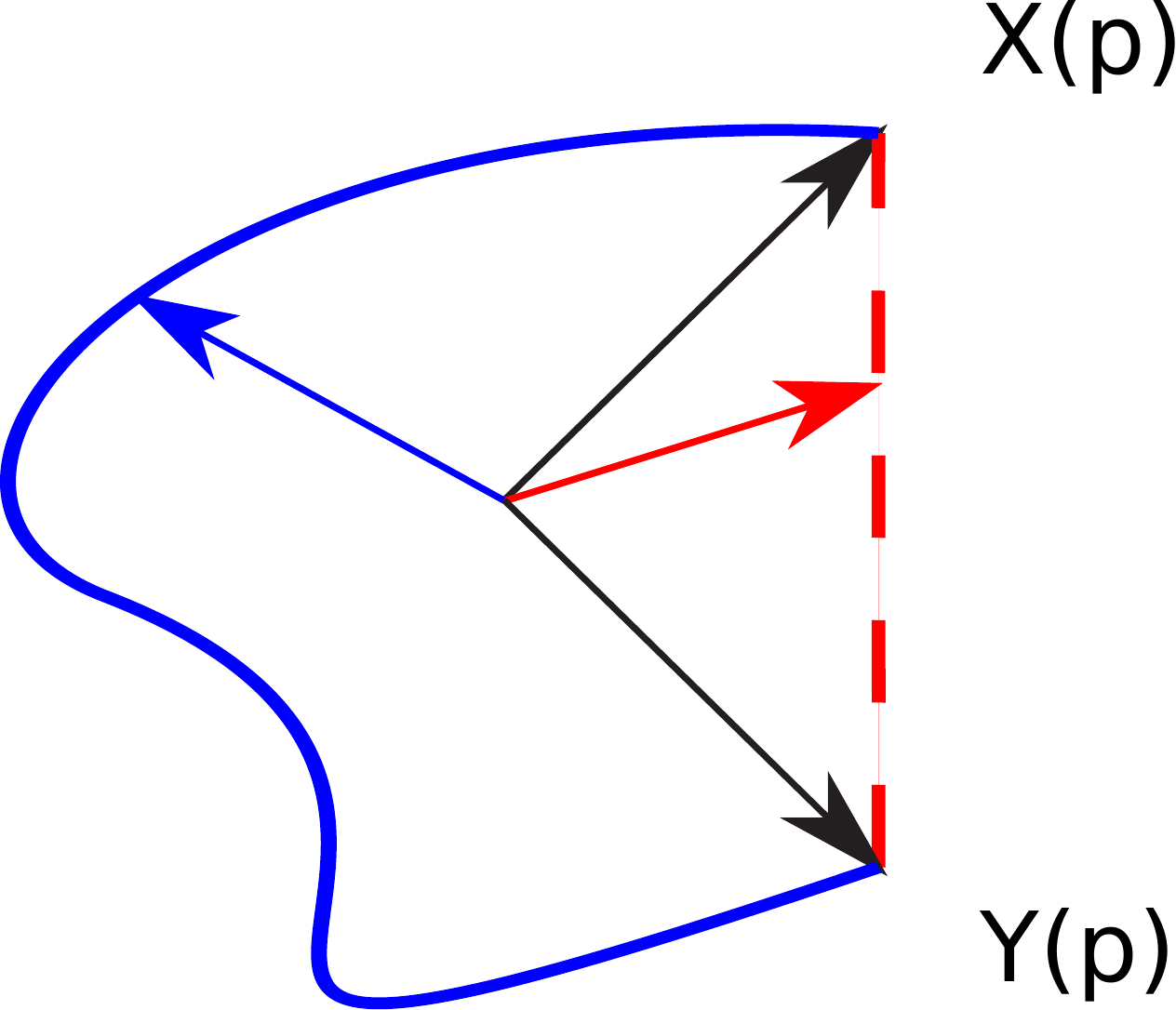}}\\
  \caption{\footnotesize{Linear (red) and nonlinear (blue) regularizations.}}
  \label{fig-nolinear}
\end{figure}

In \cite[Theorem 1]{NovaesJeffrey}, it was shown that a non monotonic linear regularization \eqref{reg-vector-field} can be transformed into a monotonic nonlinear regularization. However, in general it is not true that non monotonic linear regularizations are equivalent to monotonic nonlinear regularizations (see Theorems \ref{teob} and \ref{teoc}).

\subsection{Geometric singular perturbation theory}\label{subsec-gspt}

\noindent

In the 1970s, Neil Fenichel wrote several papers on invariant manifold theory, which allowed a rigorous study of slow-fast systems (i.e., systems of differential equations with multiple time scales). We refer to \cite{Jones, Kaper, Wiggins} for a careful introduction on slow-fast systems, as well as details of the proof given in Fenichel's original paper \cite{Fenichel}. The book \cite{Kuehn} contains introductory notions, applications and more sophisticated concepts on this subject. For applications in Biology, see \cite{Hek} and the references therein. Finally, see \cite{CardinTeixeira} for results concerning geometric singular perturbation theory for systems with many time scales.

A system of the form
\begin{equation}\label{eq-def-slowfast-1}
    \varepsilon\dot{\mathbf{x}}  =  f(\mathbf{x},\mathbf{y},\varepsilon); \ \ \
    \dot{\mathbf{y}}  =  g(\mathbf{x},\mathbf{y},\varepsilon);
\end{equation}
is called \textit{slow-fast system}, where $\mathbf{x}\in\mathbb{R}^{m}$, $\mathbf{y}\in\mathbb{R}^{n}$, $0 < \varepsilon \ll 1$ and $f:\mathbb{R}^{m}\times\mathbb{R}^{n}\times \mathbb{R}\rightarrow \mathbb{R}^{m}$, 
$g:\mathbb{R}^{m}\times\mathbb{R}^{n}\times \mathbb{R}\rightarrow \mathbb{R}^{n}$ are sufficiently smooth. The dot $\cdot$ represents the derivative of the functions $\mathbf{x}(\tau)$ and $\mathbf{y}(\tau)$ with respect to the variable $\tau$.

If we write $t = \displaystyle\frac{\tau}{\varepsilon}$, then system \eqref{eq-def-slowfast-1} becomes
\begin{equation}\label{eq-def-slowfast-2}
    \mathbf{x}' =  f(\mathbf{x},\mathbf{y},\varepsilon); \ \ \
    \mathbf{y}' = \varepsilon g(\mathbf{x},\mathbf{y},\varepsilon);
\end{equation}
in which the apostrophe ' denotes the derivative of the functions $\mathbf{x}(t)$ and $\mathbf{y}(t)$ with respect to the variable $t$. Observe that the parameter $\varepsilon = \displaystyle\frac{\tau}{t}$ represents the ratio of the time scales.

Consider equation \eqref{eq-def-slowfast-1} and set $\varepsilon = 0$. We obtain the so called \textit{slow system} given by
\begin{equation}\label{eq-def-slow-system}
    0 = f(\mathbf{x},\mathbf{y},0); \ \ \
    \dot{\mathbf{y}} = g(\mathbf{x},\mathbf{y},0).
\end{equation}

This equation is also known in the literature as \textit{reduced problem} or \textit{slow vector field}. Note that \eqref{eq-def-slow-system} is not an ODE, but it is an \textit{algebraic differential equation} (ADE).

Solutions of \eqref{eq-def-slow-system} are contained in the set
$$C_{0} = \Big{\{}(\mathbf{x},\mathbf{y})\in\mathbb{R}^{m}\times\mathbb{R}^{n}: \ f(\mathbf{x},\mathbf{y},0) = 0\Big{\}}.$$

\begin{definition}
The set $C_{0}$ is called critical set. In the case where $C_{0}$ is a manifold, $C_{0}$ is called critical manifold.
\end{definition}

On the other hand, setting $\varepsilon = 0$ in equation \eqref{eq-def-slowfast-2} we obtain the so called \textit{fast system}
\begin{equation}\label{eq-def-fast-system}
    \mathbf{x}' = f(\mathbf{x},\mathbf{y},0); \ \ \
    \mathbf{y}' = 0.
\end{equation}

System \eqref{eq-def-fast-system} is also known in the literature as \textit{layer problem}, \textit{layer equation} or \textit{fast vector field}. Moreover, the system \eqref{eq-def-fast-system} can be seen as a system of ordinary differential equations, where $\mathbf{y}\in\mathbb{R}^{n}$ is a parameter and the critical set $C_{0}$ is a set of equilibrium points of \eqref{eq-def-fast-system}.

The main goal of geometric singular perturbation theory is to study systems \eqref{eq-def-slow-system} and \eqref{eq-def-fast-system} in order to obtain information of the full system \eqref{eq-def-slowfast-1}. Observe that the systems \eqref{eq-def-slowfast-1} and \eqref{eq-def-slowfast-2} are equivalent when $\varepsilon > 0$, since they only differ by time scale.



\begin{definition}\label{delNH}
Let $\mathbf{x}_{0}\in S$, for any set $S\subset\mathbb{R}^{m+n}$. We say that $\mathbf{x}_{0}$ is normally hyperbolic if the $m\times m$ matrix $Df_{\mathbf{x}}(\mathbf{x}_{0})$ does not have eigenvalues with zero real part. The set of all normally hyperbolic points of $S$ will be denoted by $\mathcal{NH}(S)$.
\end{definition}



The nomenclature \emph{PS-singularity} and \emph{SF-singularity} is adopted in order to emphasize when $p$ is a singularity of the piecewise smooth vector field \eqref{piecewise-smooth} or a singularity of the slow-fast system \eqref{eq-def-slowfast-1} when $\varepsilon = 0$.

\subsection{Normal forms of slow-fast systems}\label{sub-sec-normal-forms-sf}

\noindent

In what follows we briefly recall some normal forms of slow-fast systems. An overview on this subject can be found in Chapter 4 of \cite{Kuehn}, and the reader can see the references therein for further details of the proofs. The normal forms of planar SF-generic transcritical and SF-generic pitchfork singularities were given in \cite{KrupaSzmolyan2}.

We say that the critical manifold $C_{0} = \{f(x,y,0) = 0\}$ has a \textit{planar SF-generic fold} (or SF-fold for short) at the origin if
\begin{equation}\label{eq-sing-fold-non-degeneracy-conditions}
\begin{split}
    f_{x}(0,0,0) = 0; \ \ f_{xx}(0,0,0) \neq 0; \\  f_{y}(0,0,0) \neq 0 \ \ \text{and} \ \ g(0,0,0)\neq 0.
\end{split}
\end{equation}

In order to obtain a \textit{SF-generic transcritical singularity} at the origin, the planar slow-fast system \eqref{eq-def-slowfast-1} must satisfy the following conditions:
\begin{equation}\label{eq-sing-transcritical-non-degeneracy-conditions}
\begin{split}
f(0,0,0) = f_{x}(0,0,0) = f_{y}(0,0,0) = 0; \\
\det\operatorname{Hes}(f) < 0; \ \ f_{xx}(0,0,0)\neq 0\neq g(0,0,0);
\end{split}
\end{equation}
where $\operatorname{Hes}(f)$ denotes the Hessian matrix of $f$, and $f_{xx}$ denotes the second derivative of $f$ with respect to the $x$ variable.

On the other hand, in order to obtain a \textit{SF-generic pitchfork singularity} at the origin we must require the following conditions:
\begin{equation}\label{eq-sing-pitchfork-non-degeneracy-conditions}
\begin{split}
f(0,0,0) = f_{x}(0,0,0) = f_{xx}(0,0,0) = f_{y}(0,0,0) = 0; \\
f_{xxx}(0,0,0) \neq 0, \ \ f_{xy}(0,0,0) \neq 0, \ \ g(0,0,0)\neq 0.
\end{split}
\end{equation}



\begin{theorem}\label{teo-fenichel-normal-form-fold}
There exists a smooth change of coordinates such that for $(x,y)$ sufficiently small the System \eqref{eq-def-slowfast-2} can be written as
\begin{description}
    \item[(a)] If the slow-fast system \eqref{eq-def-slowfast-2} satisfies the non-degeneracy conditions \eqref{eq-sing-fold-non-degeneracy-conditions} of a planar SF-generic fold:
    \begin{equation}\label{eq-teo-fenichel-normal-form-fold}
   x' = y + x^{2} + \mathcal{O}(x^{3}, xy, y^{2}, \varepsilon); \ \ \ 
   y' = \varepsilon\Big{(}\pm 1 + \mathcal{O}(x,y,\varepsilon)\Big{)}.
\end{equation}

    \item[(b)] If the slow-fast system \eqref{eq-def-slowfast-2} satisfies the non-degeneracy conditions \eqref{eq-sing-transcritical-non-degeneracy-conditions} of a SF-generic transcritical singularity: \begin{equation}\label{eq-normal-form-transcritical}
    x' = x^{2} - y^{2} + \lambda\varepsilon + \mathcal{O}(x^{3}, x^{2}y, xy^{2},y^{3},\varepsilon); \ \ \ 
    y' = \varepsilon\Big{(}1 + \mathcal{O}(x,y,\varepsilon)\Big{)}.
\end{equation}

    \item[(c)] If the slow-fast system \eqref{eq-def-slowfast-2} satisfies the non-degeneracy conditions \eqref{eq-sing-pitchfork-non-degeneracy-conditions} of a SF-pitchfork singularity: \begin{equation}\label{eq-normal-form-pitchfork}
    x' = x(y - x^{2}) + \lambda\varepsilon + \mathcal{O}(x^{2}y, xy^{2},y^{3},\varepsilon); \ \ \ 
    y' = \varepsilon\Big{(}\pm 1 + \mathcal{O}(x,y,\varepsilon)\Big{)}.
\end{equation}

\item[(d)] If $0\in C_{0}$ is a normally hyperbolic point:
\begin{equation}\label{eq-teo-fenichel-normal-form}
    \left\{
  \begin{array}{rcl}
   x_{1}' & = & \Lambda_{1}(\mathbf{x},\mathbf{y},\varepsilon)x_{1}; \\
   x_{2}' & = & \Lambda_{2}(\mathbf{x},\mathbf{y},\varepsilon)x_{2}; \\
   \mathbf{y}' & = & \varepsilon\Big{(}h(\mathbf{y},\varepsilon) + H(x_{1},x_{2},\mathbf{y},\varepsilon)(x_{1},x_{2})\Big{)};
  \end{array}
\right.
\end{equation}
where $\mathbf{x} = (x_{1},x_{2})$ is sufficiently small, $\Lambda_{j}$ (for $j = 1,2$), $h$ and $H$ are $C^{r-1}$ in all arguments. Moreover, $\Lambda_{1}(\mathbf{x},\mathbf{y},\varepsilon)$ is a matrix whose eigenvalues have positive real part, $\Lambda_{2}(\mathbf{x},\mathbf{y},\varepsilon)$ is a matrix whose eigenvalues have negative real part, and $H(x_{1},x_{2},y,\varepsilon)$ is bilinear when applied to $(x_{1},x_{2})$.
\end{description}
\end{theorem}

\section{Regularizations and typical SF-singularities}\label{mainresults}
\noindent
The relation between (\emph{linear}) regularization of piecewise smooth vector fields and slow-fast systems had led mathematicians in a new direction in the research in qualitative theory of ordinary differential equations. By applying a directional blow-up, it is possible to transform a (\emph{linearly}) regularized vector field into a slow-fast system. This approach was used for the first time in \cite{BuzziSilvaTeixeira} in the context of planar piecewise smooth vector fields, and lately by \cite{LlibreSilvaTeixeira} in the $3$-dimensional case. The $n$-dimensional case was discussed in \cite{LlibreSilvaTeixeira2}.

This study starts considering a planar piecewise smooth vector field whose discontinuity set is a smooth curve and \emph{linear} regularizations. Without loss of generality, we adopt a coordinate system such that $Z=(X(f_1,f_2),Y(g_1,g_2))$ is written as
\begin{equation}\label{eq-piecewise-smooth}
\dot{z} = Z(z)  = \dfrac{1 + \operatorname{sgn}(x)}{2}X(z) + \dfrac{1 - \operatorname{sgn}(x)}{2}Y(z),\hspace{0.4cm} z=(x,y)
\end{equation}
that is, the discontinuity set is a straight line. A linear regularization of \eqref{eq-piecewise-smooth} is the family
\begin{equation}\label{eq-piecewise-smooth-regularized}
\dot{z}= Z_{\varepsilon}(z) = \dfrac{1 + \varphi(x/\varepsilon)}{2}X(z) +\dfrac{1 - \varphi(x /\varepsilon)}{2}Y(z), 
\end{equation}
where $X=(f_{1},f_{2})$, $Y=(g_{1},g_{2})$ are applied in $z=(x,y)$. We emphasize that in this study the transition function $\varphi$ is not necessarily monotonic.

After a directional blow-up of the form $(\widetilde{x},\widetilde{y},\varepsilon) \mapsto (\varepsilon \widetilde{x}, \widetilde{y}, \varepsilon)$, one obtains the slow-fast system (dropping the tildes in order to simplify the notation)
\begin{equation}\label{eq-slow-fast-pwsvf-planar}
   \varepsilon\dot{x} = \displaystyle\frac{f_{1} + g_{1}}{2} + \varphi(x)\Bigg{(}\displaystyle\frac{f_{1} - g_{1}}{2}\Bigg{)}; \ \ \ 
   \dot{y} = \displaystyle\frac{f_{2} + g_{2}}{2} + \varphi(x)\Bigg{(}\displaystyle\frac{f_{2} - g_{2}}{2}\Bigg{)};
\end{equation}
where $f_{1}$, $f_{2}$, $g_{1}$, $g_{2}$ are applied in $(\varepsilon x,y)$. Denote the critical set of \eqref{eq-slow-fast-pwsvf-planar} by $C_{0}$. Now, we recall the definitions of sliding and sewing points presented in \cite{PanazzoloSilva}.

\begin{definition}\label{def-point-sliding}
A point $p\in\Sigma$ is a point of sliding (point of escaping) if there is an open set $U\ni p$ and a family of smooth manifolds $S_{\varepsilon}\subset U$ such that
\begin{enumerate}
    \item For each $\varepsilon$, $S_{\varepsilon}$ is invariant by system \eqref{eq-piecewise-smooth-regularized};
    \item For each compact subset $K\subset U$, the sequence $S_{\varepsilon}\cap K$ converges to $\Sigma\cap K$ as $\varepsilon \rightarrow 0$ according to Hausdorff distance.
\end{enumerate}
\end{definition}

\begin{definition}\label{def-point-sewing}
A point $p\in\Sigma$ is a point of sewing if there is an open set $U\ni p$ and local coordinates defined in $U$ such that
\begin{enumerate}
    \item $\Sigma = \{x = 0\}$;
    \item For each $\varepsilon > 0$, the vector field $\frac{\partial}{\partial x}$ is a generator of \eqref{eq-piecewise-smooth-regularized} in $U$.
\end{enumerate}
\end{definition}

Concerning linear regularizations, if the transition function is monotonic and the discontinuity set is smooth, the dynamics of the sliding vector field according to Filippov's convention is equivalent to the dynamics of the slow system associated. However, if we do not consider monotonic transition functions, one can obtain different dynamics of the (\emph{linearly}) regularized vector field and consequently different singular perturbation problems, which can lead us to different definitions of sliding (escaping) or sewing regions. See \cite{PanazzoloSilva, SilvaSarmientoNovaes} and Theorem \ref{mtheorem-A} below. \emph{Nonlinear} regularizations also lead us to different notions of sewing and sliding. See \cite{NovaesJeffrey,SilvaSarmientoNovaes}.

From the definitions, it is clear that different (linear or nonlinear) regularizations lead to different slow-fast systems, which gives rise to different sliding, escaping or sewing regions. In order to emphasize the dependency of the regularization adopted, we will call these sets as the \textit{$r$-Sliding}, \textit{$r$-Escaping} and \textit{$r$-Sewing regions}, and we will denote them as $\Sigma^{s}_{r}$, $\Sigma^{e}_{r}$ and $\Sigma^{w}_{r}$ respectively.

Consider the Filippov sliding vector field $Z^{\Sigma}$ associated to the PSVF \eqref{eq-piecewise-smooth}. Although in the literature it is only considered the dynamics of $Z^{\Sigma}$ in the Filippov sliding or escaping regions, the domain $D\Big{(}Z^{\Sigma}\Big{)}\subset\Sigma$ of $Z^{\Sigma}$ may be greater than $\Sigma^{s}\cup\Sigma^{e}$. In this sense, for our purposes, the domain $D\Big{(}Z^{\Sigma}\Big{)}$ of $Z^{\Sigma}$ is the subset of $\Sigma$ in which $Z^{\Sigma}$ is well defined.

\begin{mtheorem}\label{mtheorem-A}
Consider the PSVF \eqref{eq-piecewise-smooth} and denote its Filippov sliding vector field by $Z^{\Sigma}$, which domain is the set $D\Big{(}Z^{\Sigma}\Big{)}\subset\Sigma$. Consider linear regularization of $Z$ and let $\varphi$ be a transition function, not necessarily monotonic. Let $\Pi:\mathbb{R}^{2}\rightarrow \Sigma$ be the canonical projection. Then the following hold:
\begin{description}
    \item[(a)] If $x_{0}\in (-1,1)$ satisfies $\varphi'(x_{0}) = 0$, then the set of points $(0,y)$ such that
    $$f_{1}(0,y) + g_{1}(0,y) + \varphi(x_{0})\left(f_{1}(0,y) - g_{1}(0,y)\right) = 0$$
    is contained in $C_{0}\backslash\mathcal{NH}\Big{(}C_{0}\Big{)}$. In other words, critical points of $\varphi$ gives rise to non normally hyperbolic points of the critical set $C_{0}$ of \eqref{eq-slow-fast-pwsvf-planar}.
    \item[(b)] If $x_{0}\in (-1,1)$ satisfies $|\varphi(x_{0})| > 1$, then $\Pi\Big{(}C_{0}\Big{)}\cap\Sigma^{w}\neq \emptyset$. Moreover, $\Sigma^{s}\varsubsetneq\Sigma^{s}_{r}$. In other words, the $r$-sliding region is greater than the classical Filippov sliding region.

    \item[(c)] In the points where $f_{1}(0,y_{0})\neq g_{1}(0,y_{0})$, the dynamics in $\Sigma^{s}_{r}\cup\Sigma^{e}_{r}$ is given by $Z^{\Sigma}$. In other words, the dynamics in the classical Filippov sliding region is naturally extended to the $r$-sliding region using the Filippov sliding vector field. Moreover, $(x_{0},y_{0})$ is an SF-equilibrium point of \eqref{eq-slow-fast-pwsvf-planar} if, and only if, $(0,y_{0})$ is an equilibrium point of $Z^{\Sigma}$.
    \item[(d)] If
    $$\Pi\Big{(}C_{0}\Big{)}\cap\Bigg{(}\Sigma\backslash D\Big{(}Z^{\Sigma}\Big{)}\Bigg{)}\neq \emptyset,$$
    then $(0,y_{0})\in\Pi\Big{(}C_{0}\Big{)}\cap\Bigg{(}\Sigma\backslash D\Big{(}Z^{\Sigma}\Big{)}\Bigg{)}$ is a tangency point for both vector fields $X$ and $Y$ simultaneously, and the line $\{y = y_{0}\}$ is a component of $C_{0}$. See Figure \ref{fig-teo-a-itemd}.
\end{description}
\end{mtheorem}

\begin{proof}
\begin{description}
    \item[(a)]
    Without loss of generality, we suppose that $x_{0} = 0$. Expanding the first equation of \eqref{eq-slow-fast-pwsvf-planar} in Taylor series, one obtains
$$x' = \displaystyle\frac{1}{2}\Big{(}(f_{1} + g_{1}) + \varphi(0)(f_{1} - g_{1})\Big{)} + \displaystyle\frac{1}{2}\Big{(}\varphi'(0)(f_{1} - g_{1})\Big{)}x + \ \dots$$

A point of the form $(0,y,0)$ is normally hyperbolic if, and only if, the following conditions are satisfied:
\begin{equation}\label{eq-conditions-nh}
   (f_{1} + g_{1}) + \varphi(0)(f_{1} - g_{1})  =  0,
   \hspace{0.3cm}\varphi'(0)(f_{1} - g_{1})  \neq  0.
\end{equation}
    
Therefore, if $\varphi'(0) = 0$ (that is, $0$ is a critical point of the transition function), then $(0,y,0)$ is not normally hyperbolic.
    \item[(b)] We already know that $\Sigma^{s}\subset\Sigma^{s}_{r}$ (see Theorem 4.2, \cite{PanazzoloSilva}). Now, we prove that $\Sigma^{s}_{r}$ contains points that do not belong to $\Sigma^{s}$. Suppose without loss of generality that $\varphi(0)\neq 1$. Define the constant $a$ as
$$a = \frac{\varphi(0) + 1}{\varphi(0) - 1} \ \Leftrightarrow \ \varphi(0) = \displaystyle\frac{a + 1}{a - 1}.$$

Then the conditions \eqref{eq-conditions-nh} can be rewritten as
\begin{equation}\label{eq-conditions-sf-normal-hyperbolicity}
   g_{1}  =  af_{1}, \hspace{0.3cm}
   \varphi'(0)  \neq  0,
\end{equation}
where $f_{1}$ and $g_{1}$ are applied in $(0,y)$ and $a \neq 1$. Note that the condition $a \neq 1$ is naturally satisfied with the assumptions above. Observe that $a < 0$ if, and only if, $|\varphi(0)| < 1$. Analogously, it can be checked that $a > 0$ if, and only if, $|\varphi(0)| > 1$. 

Let $\Pi:\mathbb{R}^{2}\rightarrow \Sigma$ be the canonical projection. Assuming $|\varphi(0)| > 1$, the points of $C_{0}$ of the form $(0,y)$ such that $g_{1}(0,y) = a f_{1}(0,y)$ are projected in the Filippov sewing region $\Sigma^{w}$. 

Once again by Theorem 4.2 of \cite{PanazzoloSilva}, we have the inclusion
$$\Sigma^{s}\varsubsetneq \Pi\big{(}C_{0}\big{)}\subset \Sigma^{s}_{r}.$$

This means that $(0,y)\not\in \Sigma^{s}$ is a point of sliding, which implies that $\Sigma^{s}\varsubsetneq\Sigma^{s}_{r}$.

    \item[(c)] From the first equation of \eqref{eq-slow-fast-pwsvf-planar}, we have
    $$\varphi(x) = \displaystyle\frac{g_{1}(0,y) + f_{1}(0,y)}{g_{1}(0,y) - f_{1}(0,y)}.$$
    
    Combining this expression with the second equation of \eqref{eq-slow-fast-pwsvf-planar}, we obtain exactly the expression of $Z^{\Sigma}$. 
    \item[(d)] The domain of $Z^{\Sigma}$ is precisely the set
    $$D\Big{(}Z^{\Sigma}\Big{)} = \{(0,y)\in\Sigma \ ; \ g_{1}(0,y)\neq f_{1}(0,y)\}.$$
    
    If $(0,y_{0})\not\in D\Big{(}Z^{\Sigma}\Big{)}$ and $(0,y_{0})\in\Pi\Big{(}C_{0}\Big{)}$, then $g_{1}(0,y_{0})= f_{1}(0,y_{0})$. From the expression of $C_{0}$, $(0,y_{0})$ must be a tangency point for both $X$ and $Y$. Moreover, the equation
    $f_{1}(0,y_{0}) = 0$
    assures that the horizontal line $\{y = y_{0}\}$ is a component of the critical manifold $C_{0}$. See Figure \ref{fig-teo-a-itemd}.
\end{description}
\end{proof}

\begin{figure}[h]
  \center{\includegraphics[width=0.35\textwidth]{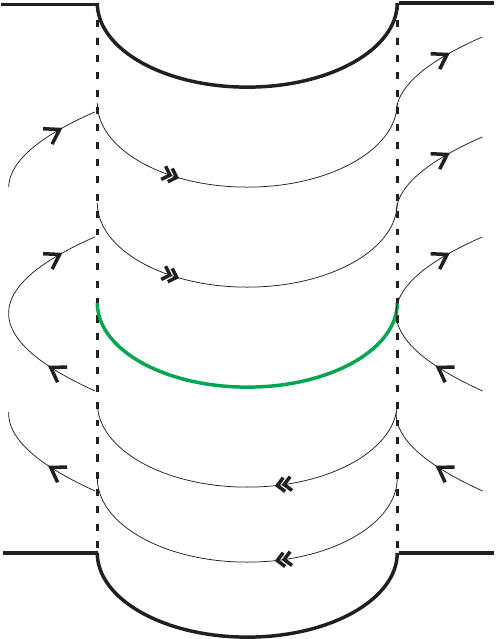}}\\
  \caption{\footnotesize{The critical set $C_{0}$, which is highlighted in green, is a horizontal line connecting two fold points. It is not possible to define dynamics in $\Sigma$ through this points using geometric singular perturbation theory.}}
  \label{fig-teo-a-itemd}
\end{figure}

Item (a) of Theorem \ref{mtheorem-A} assures that, in order to generate SF--singulari-ties with linear regularizations, we may drop the monotonicity of the transition function $\varphi$ (see Theorem \ref{teob} and Section \ref{sec:normalforms} for several examples). Moreover, $\varphi(0) = 1$ implies $f_{1}(0,y) = 0$, that is, there is a PS-tangency point between $X$ and $\Sigma$. Analogously, $\varphi(0) = -1$ implies $g_{1}(0,y) = 0$, that is, there is a PS-tangency point between $Y$ and $\Sigma$.

On the other hand, item (b) of Theorem \ref{mtheorem-A} assures that we extend the classical Filippov sliding region when the transition function satisfy $|\varphi(x_{0})| > 1$ for some $x_{0}$ in the open interval $(-1,1)$. According to item (c), the dynamics in $r$-sliding region $\Sigma^{s}_{r}$ is naturally extended using the Filippov sliding vector field (see the examples in Section \ref{sec:normalforms}, and in particular Subsection \ref{subsec-normal-form-cusp}).

Finally, item (d) says that $\Pi\big{(}C_{0}\big{)}$ is entirely contained in $D\Big{(}Z^{\Sigma}\Big{)}$, unless $C_{0}$ contains horizontal lines. This means that we can not define a sliding dynamics in $\Sigma\backslash\Big{(}D\big{(}Z^{\Sigma}\big{)}\Big{)}$ using geometric singular perturbation theory. See the examples in Subsection \ref{subsec-normal-form-fold}.

Now, we are concerned in establishing conditions that both piecewise smooth vector field and transition function must satisfy in order to generate SF-singularities. 

\begin{mtheorem}\label{teob}
Consider the PSVF \eqref{eq-piecewise-smooth} and let $\varphi$ be a transition function, not necessarily monotonic. After linear regularization and directional blow-up, it is possible to generate normally hyperbolic points, SF-fold singularities and SF-transcritical singularities. However, it is not possible to generate SF-pitchfork singularities.
\end{mtheorem}

\begin{proof}

Let $\varphi$ be a transition function (not necessarily monotonic) and $Z = (X,Y)$ be a PSVF, in which $X = (f_{1},f_{2})$ and $Y = (g_{1},g_{2})$.

The proof is given by direct computations. The idea is to compare the coefficients of the Taylor expansion of the function that defines the critical set $C_{0}$ of \eqref{eq-slow-fast-pwsvf-planar} with the expressions of the normal forms given in Subsection \ref{sub-sec-normal-forms-sf}. With this procedure, we obtain that such coefficients must satisfy the following conditions in order to generate SF-singularities:

\begin{description}
    \item[(a)] \textit{Fenichel normal form (normally hyperbolic point):}
    \begin{equation}
\begin{split}
    f_{1}(0,0) - g_{1}(0,0)  \neq  0, \ \ \ \varphi'(0)\neq 0;
\end{split}
\end{equation}

    \item[(b)] \textit{SF-generic Fold:}
\begin{equation}\label{eq-teob-sf-fold}    
 \begin{split}
    & f_{1}(0,0) - g_{1}(0,0)  \neq  0,  \ \ \ \varphi'(0)  =  0, \ \ \ \varphi''(0)  \neq  0; \\
    & \varphi(0)  =  \displaystyle\frac{g_{1}(0,0) + f_{1}(0,0)}{g_{1}(0,0) - f_{1}(0,0)}; \\
    & \big{(}f_{1,y}(0,0) + g_{1,y}(0,0)\big{)} + \varphi(0)\big{(}f_{1,y}(0,0) - g_{1,y}(0,0)\big{)}  \neq  0.
\end{split}  
\end{equation}

    \item[(c)] \textit{SF-Transcritical singularity:}
\begin{equation}\label{eq-teob-sf-transcritical}     
 \begin{split}
    & f_{1}(0,0) - g_{1}(0,0)  \neq  0,  \ \ \ \varphi'(0)  =  0, \ \ \ \varphi''(0)  \neq  0; \\
    & \varphi(0)  =  \displaystyle\frac{g_{1}(0,0) + f_{1}(0,0)}{g_{1}(0,0) - f_{1}(0,0)}; \\
    & \scriptsize\left|
  \begin{array}{cc}
\frac{1}{4}\Big{(}(f_{1} - g_{1})\varphi''(0)\Big{)} & 0 \\
    0 & \frac{1}{4}\Big{(}(1 + \varphi(0))f_{1,yy} + (1 - \varphi(0))g_{1,yy}\Big{)} \\
  \end{array}
\right| <  0; \\
    & \text{where} \ f_{1}, g_{1}, f_{1,yy} \ \text{and} \ g_{1,yy} \ \text{are computed at} \ (0,0).
\end{split}  
\end{equation}    
    
    \item[(d)] \textit{SF-Pitchfork singularity:} it is not possible to generate this kind of SF-singularity, for any transition function $\varphi$. Indeed, such a SF-singularity lead us to require $\varphi'''(0) = 0$ and $\varphi'''(0) \neq 0$ simultaneously, which is a contradiction.
\end{description}
\end{proof}

Examples of normally hyperbolic critical sets can be found in Subsection \ref{subsec-normal-form-nh}. SF-fold singularities can be seen in Subsections \ref{subsec-normal-form-fold}. Finally, SF-transcritical singularities are generated by the PS-cusp singularities. See Subsection \ref{subsec-normal-form-cusp}.

\begin{remark}
Notice that the SF-fold, SF-transcritical, and SF-pitchfork singularities are non normally hyperbolic points.
\end{remark}

\begin{corollary}
Suppose that the origin is a regular-cusp singularity of the PSVF \eqref{eq-piecewise-smooth} and let $\varphi$ be a non-monotonic transition function such that $\varphi(0)=1,$ $\varphi'(0)=0,$ and $\varphi''(0)\neq 0$. If $g_1(0,0)\varphi''(0)f_{1,yy}(0,0)>0,$ then the regularized system associated with $Z$ has a SF-transcritical singularity at origin.
\end{corollary}
\begin{proof}
Suppose that the origin is a regular-cusp singularity of the PSVF \eqref{eq-piecewise-smooth}, that is, 
\begin{itemize}
    \item $fh(0,0)=f_1(0,0)=0;$
    \item $f^2h(0,0)=f_{1,y}(0,0)f_2(0,0)=0,$ thus $f_{1,y}(0,0)=0;$
    \item $f^3h(0,0)=f_{1,yy}(0,0)(f_2(0,0))^2\neq 0,$ hence $f_{1,yy}(0,0)\neq 0;$
    \item $gh(0,0)=g_1(0,0)\neq 0;$
 \end{itemize}
 where $h(x,y)=x.$ Then, we get that
\begin{itemize}
    \item $(f_1-g_1)(0,0)=-g_1(0,0)\neq 0;$
    \item $\varphi(0)=1;$
    \item $\scriptsize\left|
  \begin{array}{cc}
\frac{1}{4}\Big{(}(f_{1} - g_{1})\varphi''(0)\Big{)} & 0 \\
    0 & \frac{1}{4}\Big{(}(1 + \varphi(0))f_{1,yy} + (1 - \varphi(0))g_{1,yy}\Big{)} \\
  \end{array}
\right|=-\frac{g_1\varphi''(0)f_{1,yy}}{8}.$
\end{itemize}
Since $\varphi'(0)=0$ and $g_1(0,0)\varphi''(0)f_{1,yy}(0,0)>0$, then Theorem \ref{teob} implies that the origin is a SF-transcritical singularity.
\end{proof}

Using the definition of a regular-fold singularity of the PSVF \eqref{eq-piecewise-smooth} and Theorem \ref{teob} we obtain the following result.

\begin{corollary}
Suppose that the origin is a regular-fold singularity of the PSVF \eqref{eq-piecewise-smooth} and let $\varphi$ be a non-monotonic transition function such that $\varphi(0)=1,$ $\varphi'(0)=0,$ and $\varphi''(0)\neq 0$. Then the regularized system associated with $Z$ has a SF-fold singularity at origin.
\end{corollary}

At some point, the reader may think that, after non monotonic linear regularization and blow-up, it is possible to generate any SF-singularity, since it is just a matter of a suitable choice of the transition function. In general, this is not true. Indeed, item (d) of Theorem \ref{teob} assures that it does not exist a transition function that generates a SF-pitchfork singularity. This leads us to consider nonlinear regularizations.

\subsection{Nonlinear regularization and SF-singularities}\label{Nonlinear regularization and singularities}
\noindent

Although the SF-pitchfork singularity cannot be obtained when the regularization is linear, it is possible to generate it if we consider the nonlinear regularization. 
In what follows, we present a version of Theorem \ref{teob} for nonlinear regularization.

\begin{mtheorem}\label{teoc}
Consider the PSVF \eqref{eq-piecewise-smooth} and let $\varphi$ be a monotonic transition function. After $\varphi$-nonlinear regularization $Z_{\varepsilon}(x,y) = \widetilde{Z}(\varphi(\frac{x}{\varepsilon}), x,y)$ and directional blow-up, it is possible to generate normally hyperbolic points, SF-fold singularities, SF-transcritical singularities and SF-pitchfork singularities.
\end{mtheorem}

\begin{proof}
Let $\varphi$ be a monotonic transition function and $Z = (X,Y)$ be a PSVF. Consider the $\varphi$-nonlinear regularization $Z_{\varepsilon}(x,y) = \widetilde{Z}(\varphi(\frac{x}{\varepsilon}), x,y)$ of $X$ and $Y$, where $\widetilde{Z}=(\widetilde{Z}^1,\widetilde{Z}^2)$. The proof is given by direct computations. The idea is to compare the coefficients of the Taylor expansion of the function $\widetilde{Z}^1(\varphi(\tilde{x}),\varepsilon\tilde{x},y)$ near $(0,0,0)$ with the expressions of the normal forms given in Subsection \ref{sub-sec-normal-forms-sf} and use that $\varphi'(t)\neq 0$ for all $t\in(-1,1)$. With this procedure, we obtain that such coefficients must satisfy the following conditions in order to generate SF-singularities:

\begin{description}
    \item[(a)] \textit{Fenichel normal form (normally hyperbolic point):}
    \begin{equation}
\begin{split}
    \widetilde{Z}^1_\lambda(\varphi(0),0,0)\neq 0;
\end{split}
\end{equation}

\item[(b)] \textit{SF-generic Fold:}
\begin{equation}\label{eq-teoc-sf-fold}        
 \begin{split}
 & \widetilde{Z}^1(\varphi(0),0,0) = 0; \ \ \ \widetilde{Z}^1_\lambda(\varphi(0),0,0) = 0; \ \ \ \widetilde{Z}^1_{\lambda\lambda}(\varphi(0),0,0)\neq 0; \\
& \widetilde{Z}^1_y(\varphi(0),0,0) \neq 0; \ \ \ \widetilde{Z}^2(\varphi(0),0,0)\neq 0.
\end{split}  
\end{equation}

\item[(c)] \textit{SF-Transcritical singularity:}
\begin{equation}\label{eq-teoc-sf-transcritical}         
 \begin{split}
& \widetilde{Z}^1(\varphi(0),0,0) = 0; \ \ \ \widetilde{Z}^1_\lambda(\varphi(0),0,0) = 0; \ \ \  \widetilde{Z}^1_{\lambda\lambda}(\varphi(0),0,0)\neq 0;\\
& \widetilde{Z}^1_y(\varphi(0),0,0)\neq 0; \ \ \ \widetilde{Z}^2(\varphi(0),0,0)\neq 0; \\
& \Big{(}\widetilde{Z}^1_{\lambda y}(\varphi(0),0,0)\Big{)}^2-\widetilde{Z}^1_{\lambda\lambda}(\varphi(0),0,0)\widetilde{Z}^1_{yy}(\varphi(0),0,0) > 0.
\end{split}  
\end{equation}    
    
\item[(d)] \textit{SF-Pitchfork singularity:}
\begin{equation}\label{eq-teoc-sf-pitchfork}     
 \begin{split}
& \widetilde{Z}^1(\varphi(0),0,0) = 0; \ \ \ \widetilde{Z}^1_\lambda(\varphi(0),0,0) = 0; \ \ \ \widetilde{Z}^1_{\lambda\lambda}(\varphi(0),0,0) = 0; \\
& \widetilde{Z}^1_y(\varphi(0),0,0) = 0; \ \ \ \widetilde{Z}^1_{\lambda\lambda\lambda}(\varphi(0),0,0) \neq 0; \ \ \ \widetilde{Z}^1_{\lambda y}(\varphi(0),0,0)\neq 0; \\
& \widetilde{Z}^2(\varphi(0),0,0)\neq 0.
\end{split}  
\end{equation}

\end{description}
\end{proof}

An example of a nonlinear regularized system with SF-pitchfork singularity is presented in the section \ref{subsec-normal-form-II-fold}. Even more, in Example \ref{exe-non-linear-reg} we provide a family of nonlinear regularizations that have this type of singularity.


\subsection{Dynamics of the regularized systems near to the SF-fold, SF-transcritical and SF-pitchfork singularities.}\label{subsec-fold-trans-pitch-reg}
\noindent

Combining Theorems \ref{teob} and \ref{teoc} and the results obtained in \cite{KrupaSzmolyan,KrupaSzmolyan2}, it is possible to describe the behavior of the orbits of the regularized systems near to the SF-fold, SF-transcritical and SF-pitchfork singularities for $\varepsilon > 0$.

\subsubsection{SF-fold case}
Let $\rho > 0$ be sufficiently small and consider a suitable interval $J\subset \mathbb{R}$. Denote by $C_{a,0}$ (resp. $C_{r,0}$) the attracting branch (resp. the repelling branch) of the critical manifold $C_0$. Suppose there exists a neighborhood $U$ of the origin such that $\Delta^{in}=\{(x, \rho^2), x\in J\}$ is a transversal section in $U$ to $C_{a,0}$ and $\Delta^{out}=\{(\rho,y), y\in\mathbb{R}\}$ is a transversal section in $U$ to the fast fibers.

The Fenichel theory assures that, for $\varepsilon > 0$, outside of a small neighbourhood of $(0,0)$, there exists two branches of slow manifolds: one is attracting ($C_{a,\varepsilon}$) and the second one is reppeling ($C_{r,\varepsilon}$). 

Let $\pi:\Delta^{in}\rightarrow\Delta^{out}$ be a transition map for the fast flow associate to \eqref{eq-piecewise-smooth-regularized}. The dynamics of the regularized system near a SF-fold singularity is established in the following Corollary (see Figure \ref{fig_fold}), which follows from Theorems \ref{teob} and \ref{teoc}, and Theorem 2.1 of \cite{KrupaSzmolyan}.
\begin{corollary}
Consider the PSVF \eqref{eq-piecewise-smooth} and let $\varphi$ be a transition function. Suppose that the origin satisfies the conditions \eqref{eq-teob-sf-fold} or \eqref{eq-teoc-sf-fold}. Then there exists $\varepsilon_{0} > 0$ such that the following statements hold for $\varepsilon\in (0, \varepsilon_{0}]:$
\begin{itemize}
    \item The manifold $C_{a,\varepsilon}$ passes through $\Delta^{out}$ at a point $(\rho, h(\varepsilon))$, where $h(\varepsilon) = \mathcal{O}(\varepsilon^{\frac{2}{3}})$.
    \item The transition map $\pi$ is a contraction with contraction rate $\mathcal{O}(e^{-\frac{c}{\varepsilon}})$, where $c$ is a positive constant.
\end{itemize}
\end{corollary}

\begin{figure}[h!]
  \center{\includegraphics[width=1.2\textwidth]{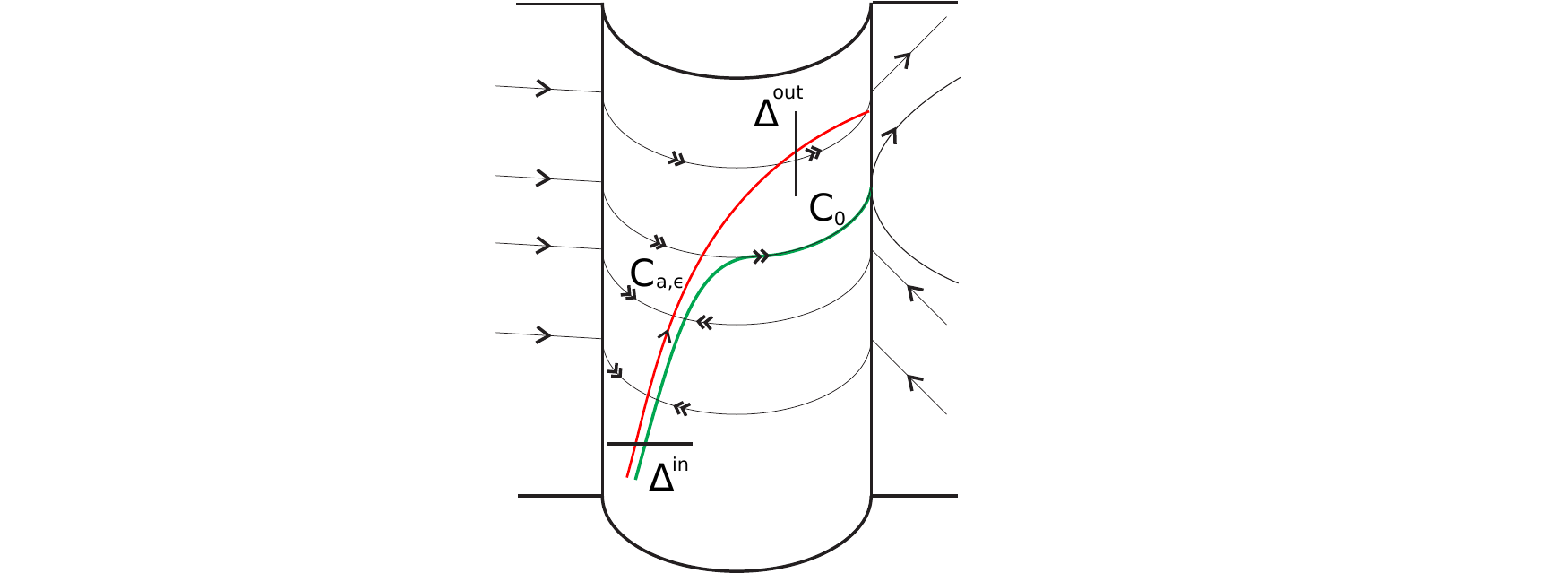}}\\
  \caption{\footnotesize{The red curve (which is not on the cylinder) represents an orbit for $\varepsilon > 0$ near a SF-fold singularity raised from a regularized PS-fold singularity of the PSVF \eqref{eq-piecewise-smooth}.}}
  \label{fig_fold}
\end{figure}

\subsubsection{SF-transcritical case}
The critical manifold $C_0$ is the union of four branches: two of them are attracting ($C^{+}_{a,0}$ and $C^{-}_{a,0}$) and the other two are repelling ($C^{+}_{r,0}$ and $C^{-}_{r,0}$), where the superscript $+$ or $-$ corresponds to the sign of the $y$ variable. 

The Fenichel theory implies that for $\varepsilon> 0$, outside of a small neighbourhood of $(0,0)$, there exist four branches of slow manifolds: two attracting ($C^{+}_{a,\varepsilon}$ and $C^{-}_{a,\varepsilon}$) and two repelling ($C^{+}_{r,\varepsilon}$ and $C^{-}_{r,\varepsilon}$). 

Now, consider a suitable neighborhood $J$ of $0\in\mathbb{R}$ and define the sections $\Delta^{in}=\{(-\rho,y), y+\rho \in J \}$, $\Delta^{out}_{e} = \{(\rho,y), y \in J \}$, and
$\Delta^{out}_{a} = \{(-\rho,y), y-\rho \in J\}$.

Let $\pi_{a}$ and $\pi_{e}$ be transition maps from $\Delta^{in}$ to $\Delta^{out}_{a}$ and $\Delta^{out}_{e}$, respectively. The dynamics of the regularized system near a SF-transcritical singularity is established in the following Corollary (see Figure \ref{fig_trans}), which is a consequence of Theorems \ref{teob} and \ref{teoc}, and the Theorem 2.1 of \cite{KrupaSzmolyan2}. We remark that the constant $\lambda$ is given in Lemma 2.1 of \cite{KrupaSzmolyan2}.

\begin{corollary}
Fix $\lambda\neq 1$. Consider the PSVF \eqref{eq-piecewise-smooth} and let $\varphi$ be a transition function. Suppose that the origin satisfies the conditions \eqref{eq-teob-sf-transcritical} or \eqref{eq-teoc-sf-transcritical}. There exists $\varepsilon_{0} > 0$ such that the following statements hold for $\varepsilon\in (0, \varepsilon_{0}]$:
\begin{itemize}
    \item If $\lambda> 1$, then the manifold $C_{a,\varepsilon}^{-}$ passes through $\Delta^{out}_{e}$ at a point $(\rho,h(\varepsilon))$, in which $h(\varepsilon) = \mathcal{O}(\sqrt{\varepsilon})$. The section $\Delta^{in}$ is mapped by $\pi_{a}$ to an interval containing $C_{a,\varepsilon}^{-}\cap\Delta^{out}_{e}$ of size $\mathcal{O}(e^{-\frac{C}{\varepsilon}})$, where $C$ is a positive constant.
    \item If $\lambda< 1$, then $\Delta^{in}$ (including the point $\Delta^{in}\cap C_{a,\varepsilon}^{-})$ is mapped by $\pi_{e}$ to an interval about $C_{a,\varepsilon}^{+}$ of size $\mathcal{O}(e^{-\frac{C}{\varepsilon}})$, where $C$ is a positive constant.
\end{itemize}
\end{corollary}

\begin{figure}[h!]
  \center{\includegraphics[width=1.2\textwidth]{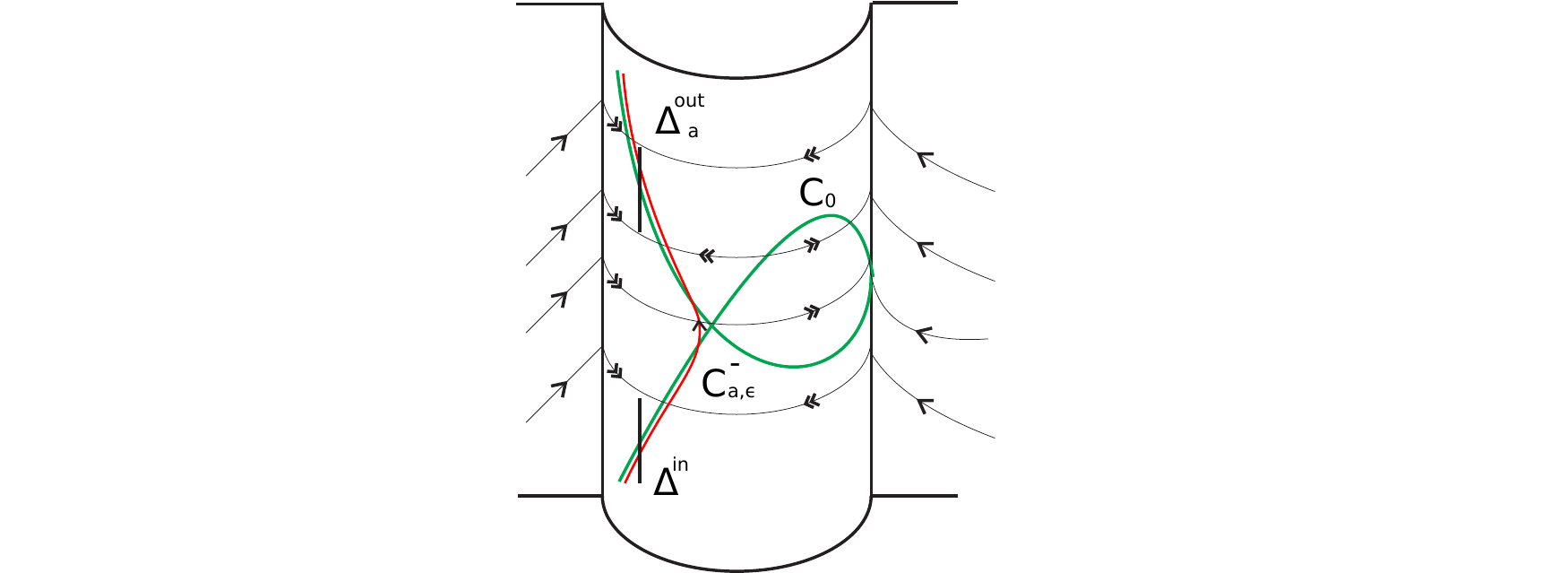}}\\
  \caption{\footnotesize{The red curve (which is not on the cylinder) represents an orbit for $\varepsilon > 0$ near a SF-transcritical singularity raised from a regularized PS-cusp singularity of the PSVF \eqref{eq-piecewise-smooth} ($\lambda < 1$).}}
  \label{fig_trans}
\end{figure}

\subsubsection{SF-Pitchfork case}


The critical manifold $C_{0}$ is the union of four branches, three stable ($C^{t}_{a}$, $C^{+}_{a}$ and $C^{-}_{a}$) and one unstable $(C^{t}_{r})$, where the superscript $\pm$ corresponds to the sign of the $x$-coordinate. 

Fenichel theory assures that, for $\varepsilon > 0$, outside of a small neighbourhood of $(0,0)$, there exist four branches of slow manifolds: three of them are attracting ($C^{t}_{a,\varepsilon}$, $C^{+}_{a,\varepsilon}$ and $C^{-}_{a,\varepsilon}$) and one repelling ($C^{t}_{r,\varepsilon}$).  

Let $J$ be a suitable neighborhood of $0\in\mathbb{R}$ and define the sections $\Delta^{t} = \{(x, -\rho), x \in J \},$ $\Delta^{+} = \{(\rho , y), y - \rho^{2} \in J \}$, and 
$\Delta^{-} = \{(-\rho , y), y - \rho^{2} \in J \}$.

Let $\pi^{t}$ be the transition map from $\Delta^{t}$ to $\Delta^{+}\cup\Delta^{-}$. This map is defined when $g^*_{x,y,\varepsilon}>0$ (expression given in Lemma 4.1 ($(4.4b)$) of \cite{KrupaSzmolyan2}). In the case of $g^*_{x,y,\varepsilon}<0$ we define transition maps $\pi^{\pm}:\Delta^{\pm}\rightarrow\Delta^{t}$. The dynamics of the nonlinear regularized system near a SF-pitchfork singularity is established in the following Corollary (see Figure \ref{fig_pitchfork}), which is a consequence of Theorems \ref{teoc} and the Theorem 4.1 of \cite{KrupaSzmolyan2}.  We remark that the constant $\lambda$ is given in Lemma 4.1 of \cite{KrupaSzmolyan2}.

\begin{corollary}
 Fix $\lambda\neq 0$.  Let $\varphi$ be a  monotonic transition function and consider $Z_{\varepsilon}$ a $\varphi$-nonlinear regularization of $X$ and $Y$. Suppose that the origin satisfies the conditions \eqref{eq-teoc-sf-pitchfork}, then there exists $\varepsilon_{0} > 0$ such that the following statements hold for $\varepsilon\in (0, \varepsilon_{0}]$:
 \begin{itemize}
     \item If $g^*_{x,y,\varepsilon}>0$ and $\lambda > 0$, then $\Delta^{t}$ (including the point $\Delta^{t}\cap C^{t}_{a,\varepsilon})$ is mapped by $\pi^{t}$ to an interval near $\Delta^{+}\cap C^{+}_{a,\varepsilon}$ of size $\mathcal{O}(e^{-\frac{C}{\varepsilon}})$, where $C$ is a positive constant.
     \item If $g^*_{x,y,\varepsilon}>0$ and $\lambda< 0$, then $\Delta^{t}$ (including the point $\Delta^{t} \cap C^{t}_{a,\varepsilon}$) is mapped by $\pi^{t}$ to an interval near $\Delta^{-}\cap S^{-}_{a,\varepsilon}$ of size $\mathcal{O}(e^{-\frac{C}{\varepsilon}})$, where $C$ is a positive constant.
     \item If $g^*_{x,y,\varepsilon}<0$, then $\Delta^{+}$ and $\Delta^{-}$ are mapped by $\pi^{+}$ and $\pi^{-}$, respectively, to intervals near $C^{t}_{a,\varepsilon}\cap\Delta^{t}$ of size $\mathcal{O}(e^{-\frac{C}{\varepsilon}})$, where $C$ is a positive constant.
 \end{itemize}
\end{corollary}

\begin{figure}[h!]
  \center{\includegraphics[width=0.4\textwidth]{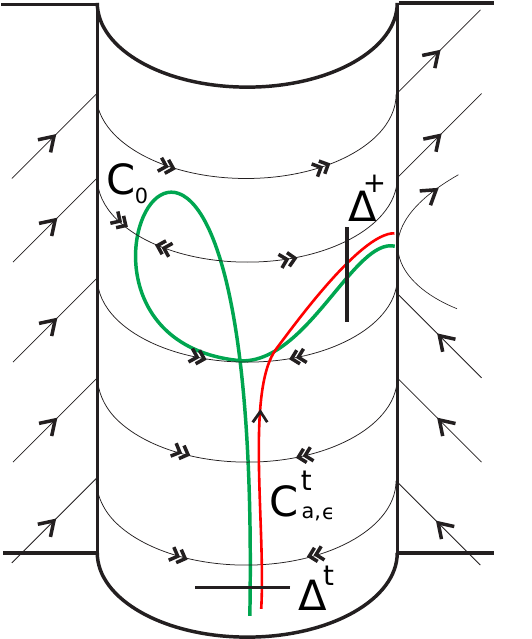}}\\
  \caption{\footnotesize{The red curve represents an orbit for $\varepsilon > 0$ of the nonlinearly regularized system $Z_\varepsilon(x,y)=\widetilde{Z}(\phi(\frac{x}{\varepsilon}),x,y),$ where $\widetilde{Z}(\lambda,x,y)=((x+\lambda-1)y-\lambda^3+e^{\lambda x+y}-1,1)$ near a SF-pitchfork singularity ($\lambda>0$).} Notice that the red curve is not on the cylinder} 
  \label{fig_pitchfork}
\end{figure}
\section{Normal forms of piecewise smooth vector fields and Slow-fast systems}\label{sec:normalforms}
\noindent
In this section we study normal forms of piecewise smooth vector fields (which can be found in \cite{GuardiaSearaTeixeira, Kuznetsov}) and we investigate which SF-normal form can be generated from each PS-normal form. The strategy is to use the ideas of Appendix \ref{sec-trans-func} in order to construct a suitable transition function that satisfies the conditions stated in Theorems \ref{teob} and \ref{teoc}. Structurally stable PSVF's and some codimention-1 bifurcations are considered.

We emphasize that, in our notation, $\Sigma^{s}$, $\Sigma^{e}$ and $\Sigma^{w}$ denote the classical Filippov sliding, escaping and sewing regions, respectively. On the other hand, we adopt the notation $\Sigma^{s}_{r}$, $\Sigma^{e}_{r}$ and $\Sigma^{w}_{r}$ in order to emphasize that such regions depend on the (linear or nonlinear) regularization adopted.
\subsection{Structurally stable PS-normal forms}\label{subsec-normal-form-nh}
Here we analyze normal forms of structurally stable piecewise smooth vector fields.
\begin{example}
A trivial example of piecewise smooth vector field that gives rise to a normally hyperbolic critical manifold after (\emph{linear}) regularization and blow-up is
\begin{equation}\label{eq-exe-piecewise-smooth-normal-hyper-1}
Z(x,y) = \left\{
             \begin{array}{ccccc}
               X(x,y) & = & \Big{(}\alpha, 1\Big{)}, & \text{if} & x > 0; \\
               Y(x,y) & = & \Big{(}\beta,1\Big{)}, & \text{if} & x < 0;
             \end{array}
           \right.
\end{equation}
with $\operatorname{sgn}(\alpha)\neq\operatorname{sgn}(\beta)$ and $\varphi(t)$ is a monotonic transition function given by
\begin{equation}\label{exe-monotonic-trans-func}
\varphi(t) = \left\{
  \begin{array}{rcl}
    -1, & \text{if} &  t\leq -1; \\
     \omega t^{4} -\frac{t^{3}}{2} - 2\omega t^{2} + \frac{3t}{2} + \omega, & \text{if} & -1 \leq t \leq 1; \\
    1, & \text{if} & t\geq 1;
  \end{array}
\right.
\end{equation}
where $\omega = \frac{\beta + \alpha}{\beta - \alpha}$. For example, if $\alpha = 1$ and $\beta = -1$, we obtain a repelling normally hyperbolic critical manifold. See Figure \ref{fig-nh-escape}.

\begin{figure}[h!]
  \center{\includegraphics[width=0.48\textwidth]{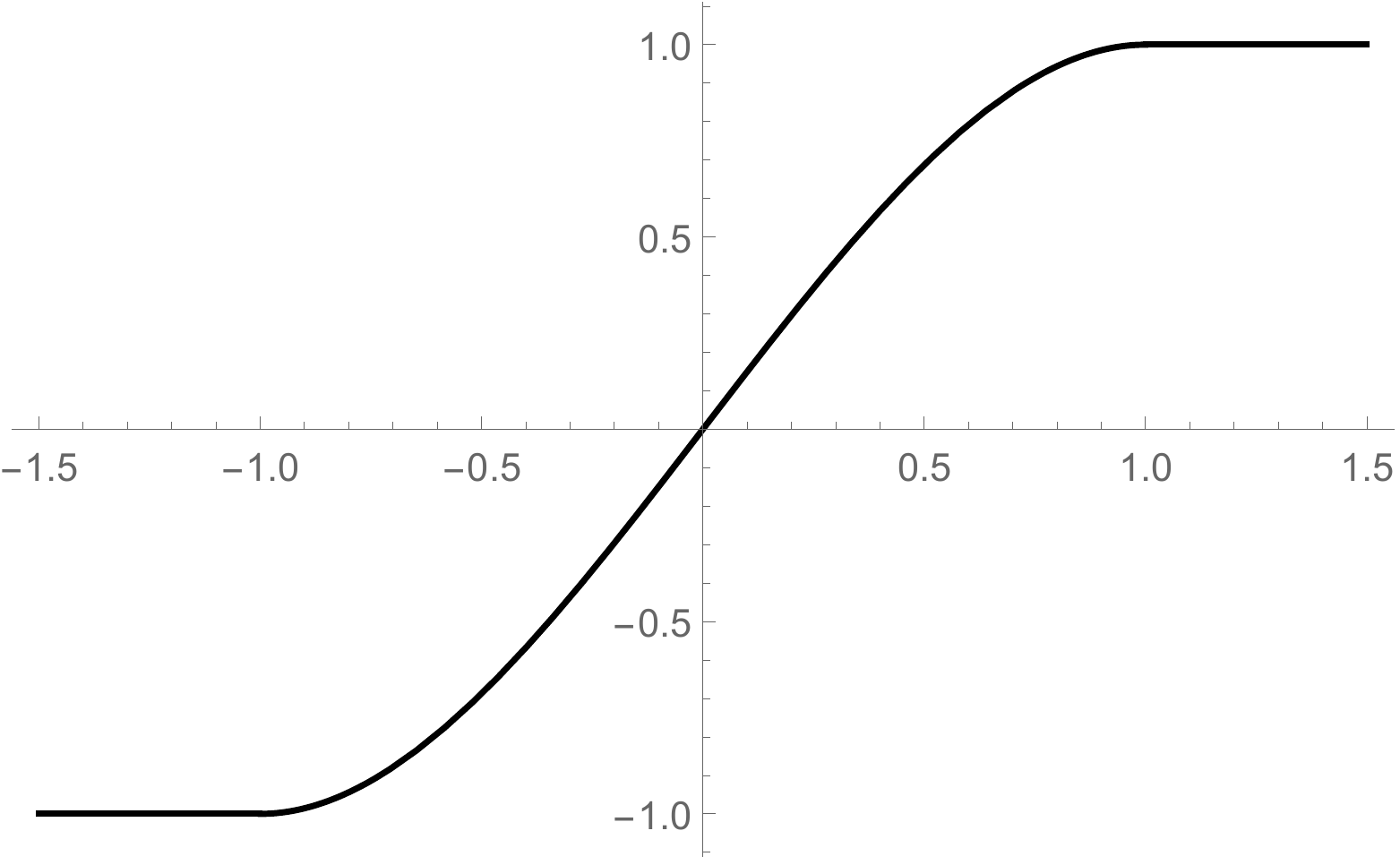}\hspace{0.7cm}\includegraphics[width=0.35\textwidth]{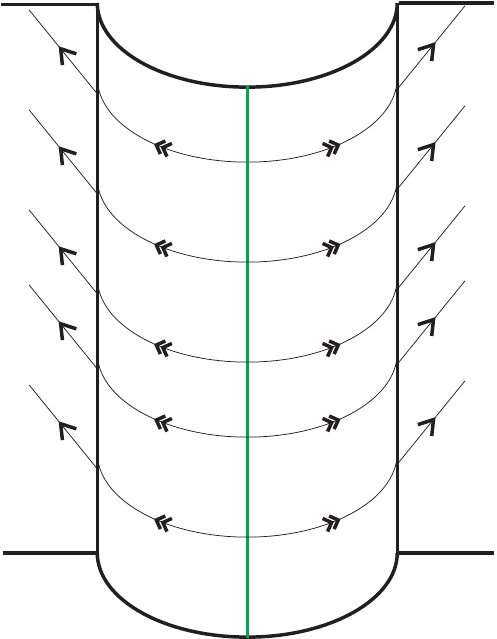}}\\
  \caption{\footnotesize{Piecewise smooth vector field \eqref{eq-exe-piecewise-smooth-normal-hyper-1} regularized with the transition function \eqref{exe-monotonic-trans-func}. The critical manifold is highlighted in green.}}
  \label{fig-nh-escape}
\end{figure}

\end{example}

\begin{example}
If the transition function $\varphi$ is non monotonic, it is possible to generate a slow-fast system having normally hyperbolic critical manifold from a PSVF such that $\Sigma^{w} = \Sigma$. Indeed, using the ideas of Appendix \ref{sec-trans-func}, in the interval $[-1,1]$ we impose the conditions
\begin{equation}\label{exe-trans-function-construc}
\begin{split}
\varphi'(-1) = 0; \ \ \varphi'(1) = 0; \ \ \varphi'(0) \neq 0;  \\
\varphi(-1) = -1; \ \ \varphi(1) = 1; \ \ |\varphi(0)| > 1; 
\end{split}
\end{equation}
and then we get the non monotonic transition function \eqref{exe-monotonic-trans-func}. If we consider linear regularization, we have to impose $|\varphi(0)| = |\omega| > 1$. For example, consider the PSVF
\begin{equation}\label{eq-exe-piecewise-smooth-nh-sew}
Z(x,y) = \left\{
             \begin{array}{ccccc}
               X(x,y) & = & \Big{(}1,0\Big{)}, & \text{if} & x > 0; \\
               Y(x,y) & = & \Big{(}2,0\Big{)}, & \text{if} & x < 0.
             \end{array}
           \right.
\end{equation}

The slow-fast system obtained after (\emph{linear}) regularization and blow-up is
\begin{equation}\label{eq-exe-piecewise-smooth-nh-sew-slow-fast}
               \varepsilon\dot{x} = -\frac{3x}{4} + 3x^{2} + \frac{x^{3}}{4} - \frac{3x^{4}}{2}; \ \ \
               \dot{y} = 0.
\end{equation}

Note that $|\varphi(0)| = |\omega| = 3$. Concerning the PSVF \eqref{eq-exe-piecewise-smooth-nh-sew}, it is easy to check that the discontinuity locus $\Sigma$ is a Filippov sewing region. However, the critical manifold of the slow-fast system \eqref{eq-exe-piecewise-smooth-nh-sew-slow-fast} has two normally hyperbolic components. By Theorem \ref{mtheorem-A}, since $\varphi(0) = 3 > 1$, we obtain a $\varphi$-sliding region. In particular, $\Sigma = \Sigma^{s}_{r}$. See Figure \ref{fig-nh-sewing}.

\begin{figure}[h!]
  \center{\includegraphics[width=0.48\textwidth]{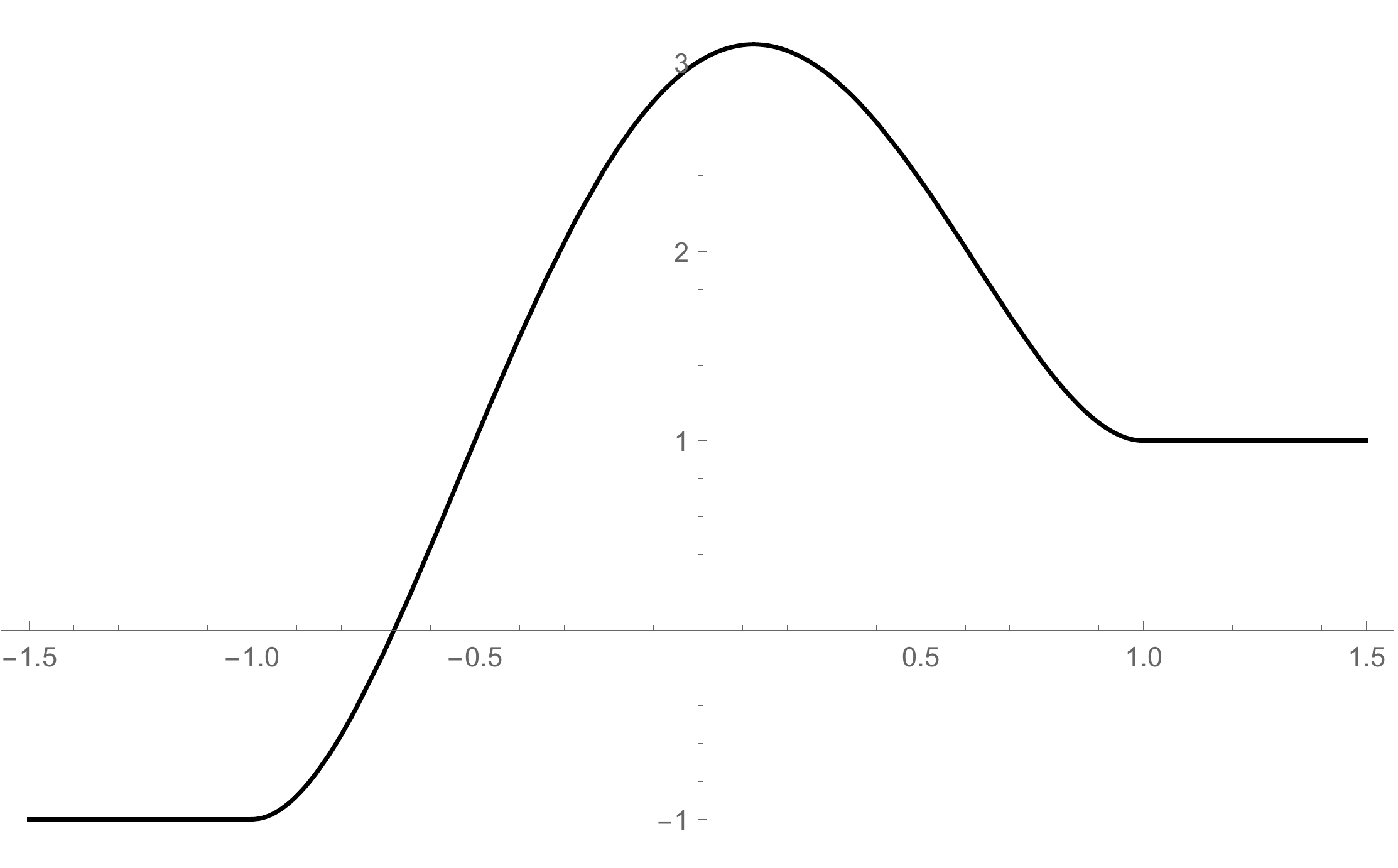}\hspace{0.7cm}\includegraphics[width=0.40\textwidth]{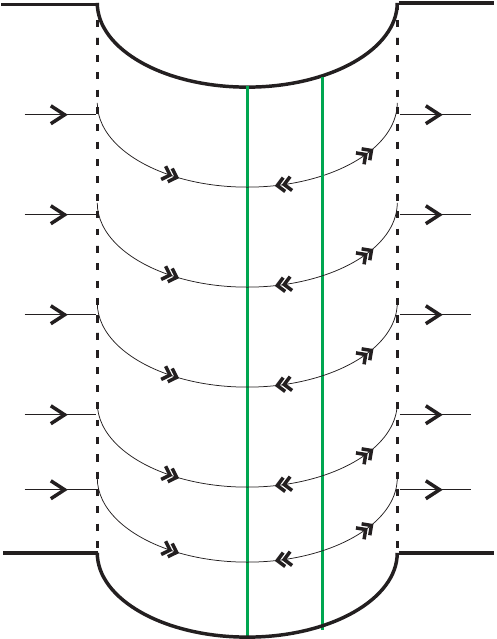}}\\
  \caption{\footnotesize{Graphic of the transition function given by \eqref{exe-monotonic-trans-func} with $\omega = 3$ (left) and regularization of the piecewise smooth vector field \eqref{eq-exe-piecewise-smooth-nh-sew} (right). The critical manifold is highlighted in green.}}
  \label{fig-nh-sewing}
\end{figure}

\end{example}

\begin{example}
Let $Z$ be the piecewise smooth vector field given by
\begin{equation}\label{eq-exe-piecewise-smooth-nh-fold}
Z(x,y) = \left\{
             \begin{array}{ccccc}
               X(x,y) & = & \Big{(}\alpha y, 1\Big{)}, & \text{if} & x > 0; \\
               Y(x,y) & = & \Big{(}\beta, 0 \Big{)}, & \text{if} & x < 0;
             \end{array}
           \right.
\end{equation}
where $\alpha \neq 0 \neq \beta$ and let $\varphi$ be the transition function given by
\begin{equation}\label{exe-non-monotonic-trans-func-NH-FOLD}
\varphi(t) = \left\{
  \begin{array}{rcl}
    -1, & \text{if} &  t\leq -1; \\
      t^{4} - \frac{t^{3}}{2} - 2t^{2} + \frac{3t}{2} + 1, & \text{if} & -1 \leq t \leq 1; \\
    1, & \text{if} & t\geq 1.
  \end{array}
\right.
\end{equation}

After (\emph{linear}) regularization and blow-up, we obtain the slow-fast system
\begin{equation}\label{eq-exe-piecewise-smooth-nh-FOLD-slow-fast}
\left\{
             \begin{array}{rcl}
               \varepsilon\dot{x} & = & \frac{1}{4}\Big{(}4 \alpha y + x(x-1)^{2}(2x + 3) (\alpha y-\beta)\Big{)}; \\
               \dot{y} & = & 1 + \frac{3x}{4} - x^{2} - \frac{x^{3}}{4} + \frac{x^{4}}{2}.
             \end{array}
           \right.
\end{equation}

Observe that $\varphi'\Big{(}\frac{3}{8}\Big{)} = 0$. Therefore,  the critical manifold of \eqref{eq-exe-piecewise-smooth-nh-FOLD-slow-fast} has a non normally hyperbolic point for $x = \frac{3}{8}$. It can be shown that such point is a SF-fold. By Theorem \ref{mtheorem-A}, since $\varphi\Big{(}\frac{3}{8}\Big{)} > 1$, it follows that $\Sigma^{s} \varsubsetneq \Sigma^{s}_{r}$. See Figure \ref{fig-nh-fold}.

\begin{figure}[h!]
  \center{\includegraphics[width=0.48\textwidth]{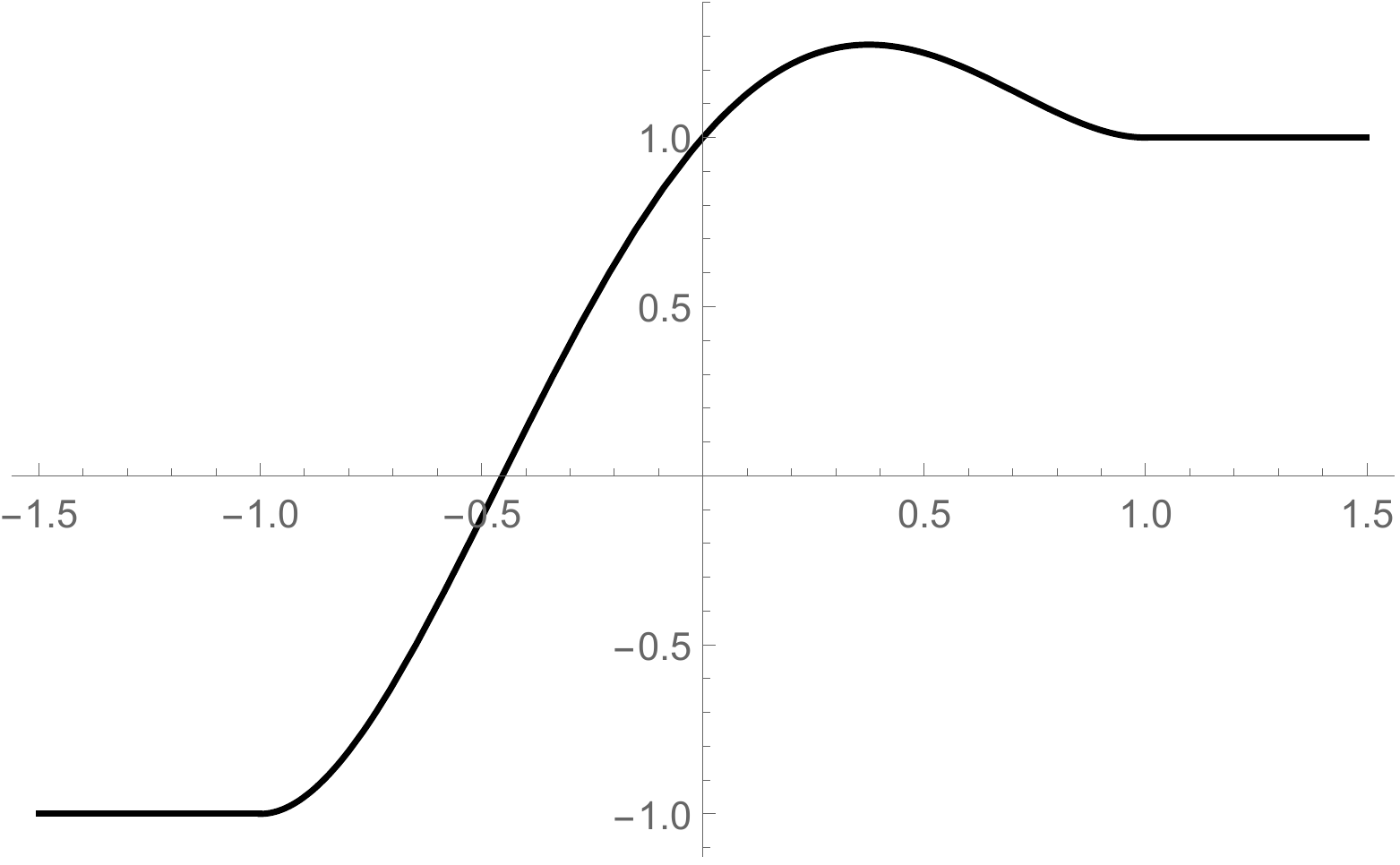}\hspace{0.3cm}\includegraphics[width=0.40\textwidth]{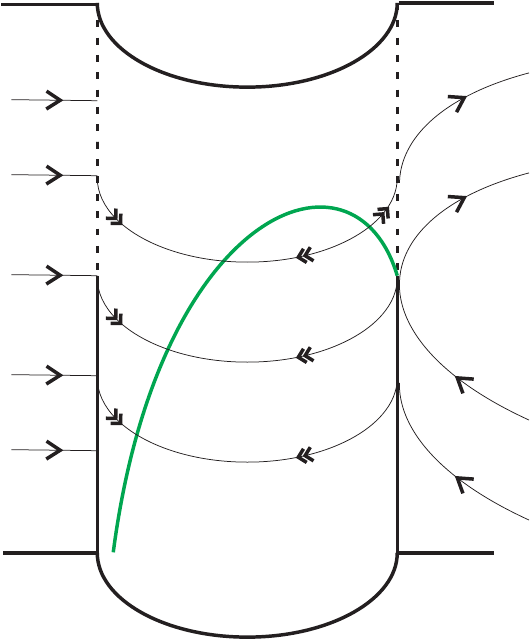}}\\
  \caption{\footnotesize{Non monotonic transition function \eqref{exe-non-monotonic-trans-func-NH-FOLD} (left) and regularization of the piecewise smooth vector field \eqref{eq-exe-piecewise-smooth-nh-fold} (right). The critical manifold is highlighted in green.}}
  \label{fig-nh-fold}
\end{figure}

\end{example}

\begin{example}
This example concerns a PS-singularity that is a hyperbolic equilibrium point of the sliding vector field $Z^{\Sigma}(x,y)$. Let $Z$ be the piecewise smooth vector field
\begin{equation}\label{eq-exe-piecewise-smooth-pseudo-sing}
Z(x,y) = \left\{
             \begin{array}{ccccc}
               X(x,y) & = & \Big{(}-1, y\Big{)}, & \text{if} & x > 0; \\
               Y(x,y) & = & \Big{(}1,y\Big{)}, & \text{if} & x < 0;
             \end{array}
           \right.
\end{equation}
and $\varphi(t)$ is a monotonic transition function given by \eqref{exe-monotonic-trans-func}, with $\omega = 0$. After (\emph{linear}) regularization and blow-up, we obtain the slow-fast system
\begin{equation}\label{eq-exe-piecewise-smooth-nh-pseudo-sing-slow-fast}
               \varepsilon\dot{x} = \frac{x^{3}}{2}-\frac{3x}{2}; \ \ \
               \dot{y} = y.
\end{equation}

See Figure \ref{fig-nh-pseudo-sing}.

\begin{figure}[h]
  \center{\includegraphics[width=0.40\textwidth]{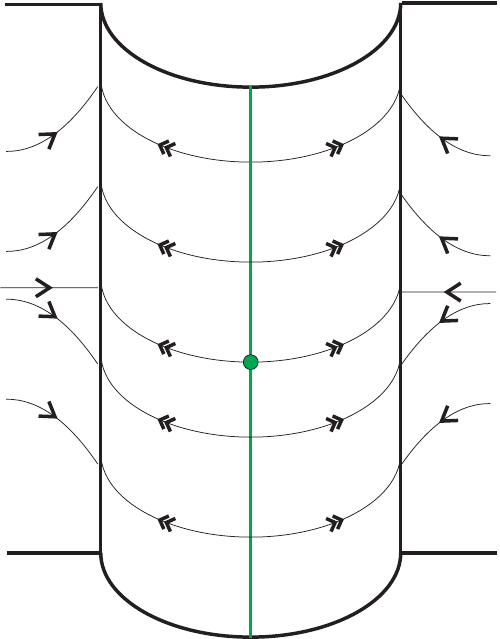}}\\
  \caption{\footnotesize{Regularized piecewise smooth vector field \eqref{eq-exe-piecewise-smooth-pseudo-sing} using the transition function \eqref{exe-monotonic-trans-func} with $\omega = 0$. The critical manifold is highlighted in green.}}
  \label{fig-nh-pseudo-sing}
\end{figure}

\end{example}

In what follows, we start the study of (linear and nonlinear) regularizations of codimention-1 bifurcations of PSVF's.

\subsection{Codimension 1 bifurcation: PS-cusp singularity}\label{subsec-normal-form-cusp}
\noindent

\begin{example}
Consider the normal form of a PS-cusp singularity
\begin{equation}\label{eq-exe-piecewise-smooth-pwsvf-cusp-case-1}
Z(x,y) = \left\{
             \begin{array}{ccccc}
               X(x,y) & = & \Big{(}-y^{2} + \lambda, 1\Big{)}, & \text{if} & x > 0; \\
               Y(x,y) & = & \Big{(}1, 1\Big{)}, & \text{if} & x < 0.
             \end{array}
           \right.
\end{equation}

For $\lambda = 0$, the origin is a PS-cusp singularity and $\Sigma^{s} = \Sigma\backslash\{0\}$. For $\lambda < 0$, $\Sigma^{s} = \Sigma$. Finally, for $\lambda > 0$ the points $(0,\pm\sqrt{\lambda})$ are PS-folds of $Z$, $\Sigma^{w} = \{-\sqrt{\lambda} < y < \sqrt{\lambda}\}$ and $\Sigma^{s} = \Sigma\backslash\Sigma^{w}$. See figure \ref{fig-bif1-cusp-sliding1}.

Combining the ideas of Appendix \ref{sec-trans-func} and the conditions given by Theorem \ref{teob}, we construct a transition function $\varphi$ given by
\begin{equation}\label{exe-non-monotonic-trans-func-cusp-case-1}
\varphi(t) = \left\{
  \begin{array}{rcl}
    -1, & \text{if} &  t\leq -1; \\
      -\frac{3 t^{5}}{2}+t^{4}+\frac{5 t^{3}}{2}-2 t^{2}+1, & \text{if} & -1 \leq t \leq 1; \\
    1, & \text{if} & t\geq 1;
  \end{array}
\right.
\end{equation}
in which $t_{0} = 0$ and $t_{1} = \frac{8}{15}$ are local maximum and minimum, respectively. See Figure \ref{fig-cusp-nao-preserva-transcritica-caso-1-trans-func}.

\begin{figure}[h]
  \center{\includegraphics[width=0.48\textwidth]{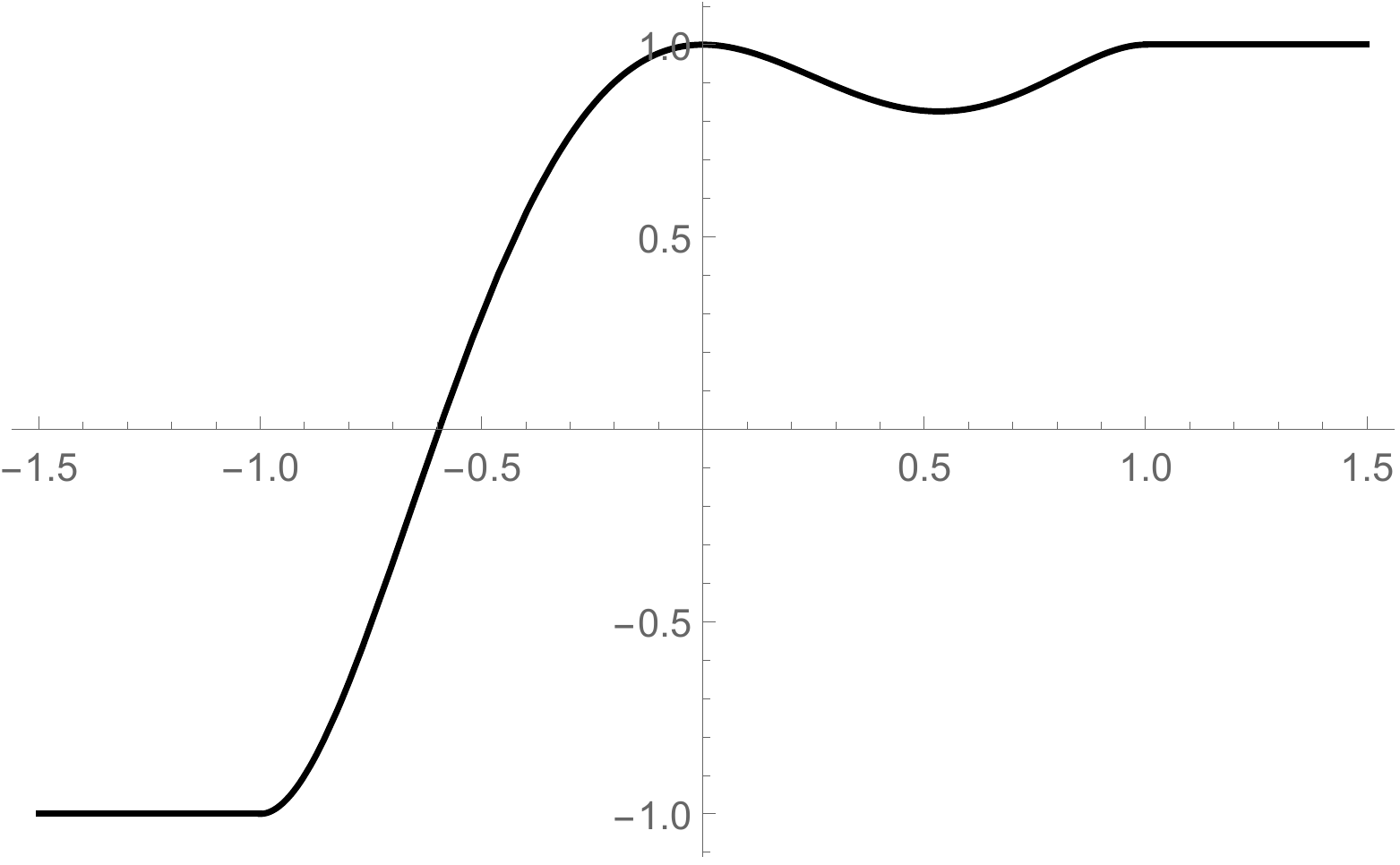}}\\
  \caption{\footnotesize{Graphic of the non monotonic transition function \eqref{exe-non-monotonic-trans-func-cusp-case-1}.}}
  \label{fig-cusp-nao-preserva-transcritica-caso-1-trans-func}
\end{figure}

After regularization and blow-up, one obtains the slow-fast system
\begin{equation}\label{eq-exe-piecewise-smooth-cusp-case-1-slow-fast}
\left\{
             \begin{array}{rcl}
               \varepsilon\dot{x} & = & \frac{1}{4} \Bigg{(}\lambda \Big{(}4- x^{2}(x-1)^{2}(3x + 4)\Big{)}+ x^{2}(x-1)^{2}(3x + 4)\left(y^{2}+1\right)-4 y^{2}\Bigg{)}; \\
               \dot{y} & = & 1.
             \end{array}
           \right.
\end{equation}

Observe that for $x = 0$ and $x = \frac{8}{15}$, the critical manifold presents non normally hyperbolic points. In particular, the origin is a transcritical singularity that is destroyed for $\lambda\neq 0$. Observe that $\varphi(t) \leq 1$. See Figure \ref{fig-bif1-cusp-sliding1}.

\begin{figure}[h]
  \center{\includegraphics[width=1\textwidth]{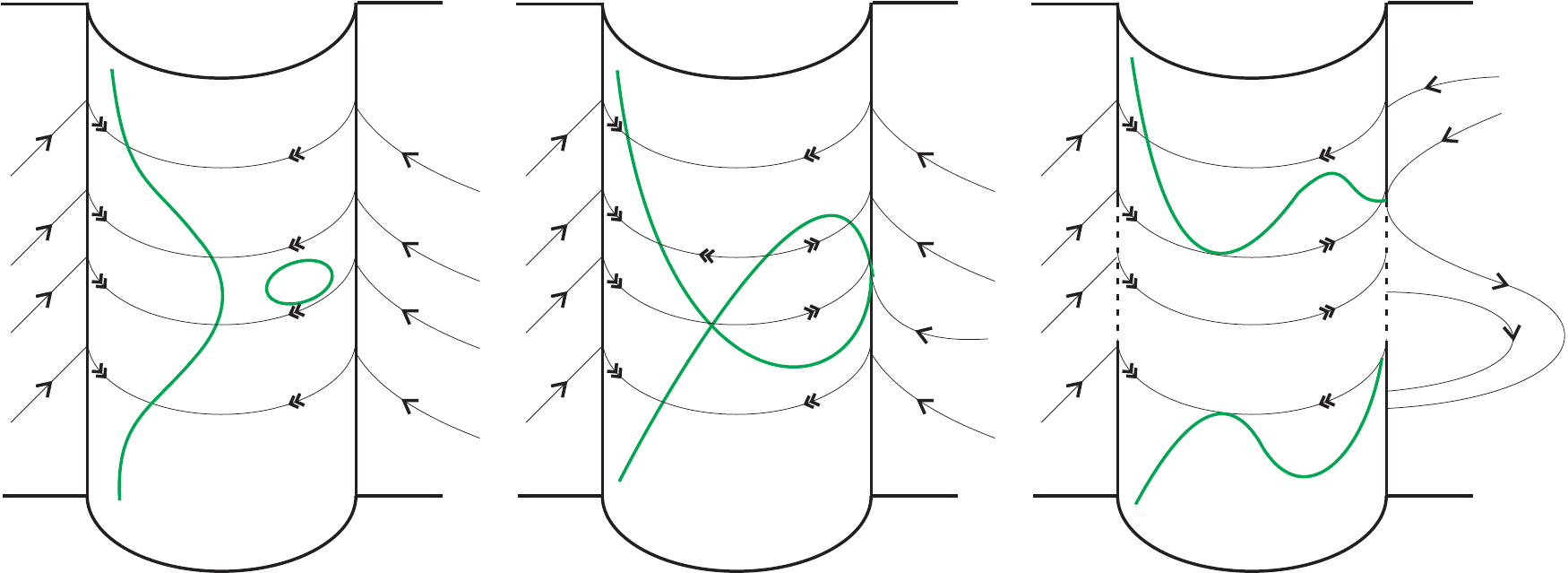}}\\
  \caption{\footnotesize{Regularized piecewise smooth vector field \eqref{eq-exe-piecewise-smooth-pwsvf-cusp-case-1} using the transition function \eqref{exe-non-monotonic-trans-func-cusp-case-1} for $\lambda < 0$ (left), $\lambda = 0$ (center) and $\lambda > 0$ (right). The green curve is the critical set.}}
  \label{fig-bif1-cusp-sliding1}
\end{figure}

\end{example}

In the last example, we have destroyed the SF-transcritical singularity by perturbing the parameter $\lambda$. In general, we have the following result.

\begin{proposition}\label{prop:sing_trans_bif}
Let $Z_\lambda$, $\lambda\in(-1,1)$ be the piecewise smooth vector field \eqref{eq-exe-piecewise-smooth-pwsvf-cusp-case-1} and let $\varphi$ be a non-monotonic transition function such that $\varphi(0)=1,$ $\varphi'(0)=0,$ and $\varphi''(0)<0$. Then the linearly regularized system associated to $Z_\lambda$ has a SF-transcritical singularity at origin for $\lambda=0$, which disappears for $\lambda\neq 0$.
\end{proposition}
\begin{proof}
Consider the piecewise smooth vector field \eqref{eq-exe-piecewise-smooth-pwsvf-cusp-case-1} with $\lambda\in(-1,1)$. Then $f_1(x,y)=-y^2+\lambda$, $f_2(x,y)=1$, $g_1(x,y)=1,$ and $g_2(x,y)=1.$ Thus, we get that
\begin{itemize}
    \item $(f_1-g_1)(0,0)=\lambda-1\neq 0;$
    \item $\varphi(0)=\frac{1+\lambda}{1-\lambda}=1$ if, and only if, $\lambda=0.$
    \item $\scriptsize\left|
  \begin{array}{cc}
\frac{1}{4}\Big{(}(f_{1} - g_{1})\varphi''(0)\Big{)} & 0 \\
    0 & \frac{1}{4}\Big{(}(1 + \varphi(0))f_{1,yy} + (1 - \varphi(0))g_{1,yy}\Big{)} \\
  \end{array}
\right|=\frac{\varphi''(0)}{4} <  0.$
\end{itemize}
Since $\varphi'(0)=0$ and $\varphi''(0)<0$, then Theorem \ref{teob} implies that the origin is a transcritical singularity provide that $\lambda=0.$
\end{proof}

In the next example we perturb the PSVF and the transition function simultaneously, in such a way that the SF-transcritical singularity persists for $\lambda \neq 0$.

\begin{example}
Consider once again the piecewise smooth vector field \eqref{eq-exe-piecewise-smooth-pwsvf-cusp-case-1}, whose bifurcation diagram is given in the Figure \ref{fig-bif1-cusp-sliding1}. In this example we adopt the transition function
\begin{equation}\label{exe-non-monotonic-trans-func-cusp-case-1-preserva-transcritica}
\varphi_\lambda(t) = \left\{
  \begin{array}{rcl}
    -1, & \text{if} &  t\leq -1; \\
      -\frac{3 t^{5}}{2} + \frac{(\lambda+1) t^{4}}{1-\lambda}+\frac{5 t^{3}}{2} -\frac{2 (\lambda + 1) t^{2}}{1-\lambda}+\frac{\lambda+1}{1-\lambda}, & \text{if} & -1 \leq t \leq 1; \\
    1, & \text{if} & t\geq 1;
  \end{array}
\right.
\end{equation}
which is a perturbation of the transition function \eqref{exe-non-monotonic-trans-func-cusp-case-1} and its bifurcation diagram for small $\lambda$ can be seen in Figure \ref{fig-cusp-preserva-transcritica-caso-1-trans-func}. For $t = 0$ and $t = -\frac{8 (\lambda+1)}{15 (\lambda-1)}$, the derivative of $\varphi_\lambda$ is zero.
\begin{figure}[h]
  \center{\includegraphics[width=0.315\textwidth]{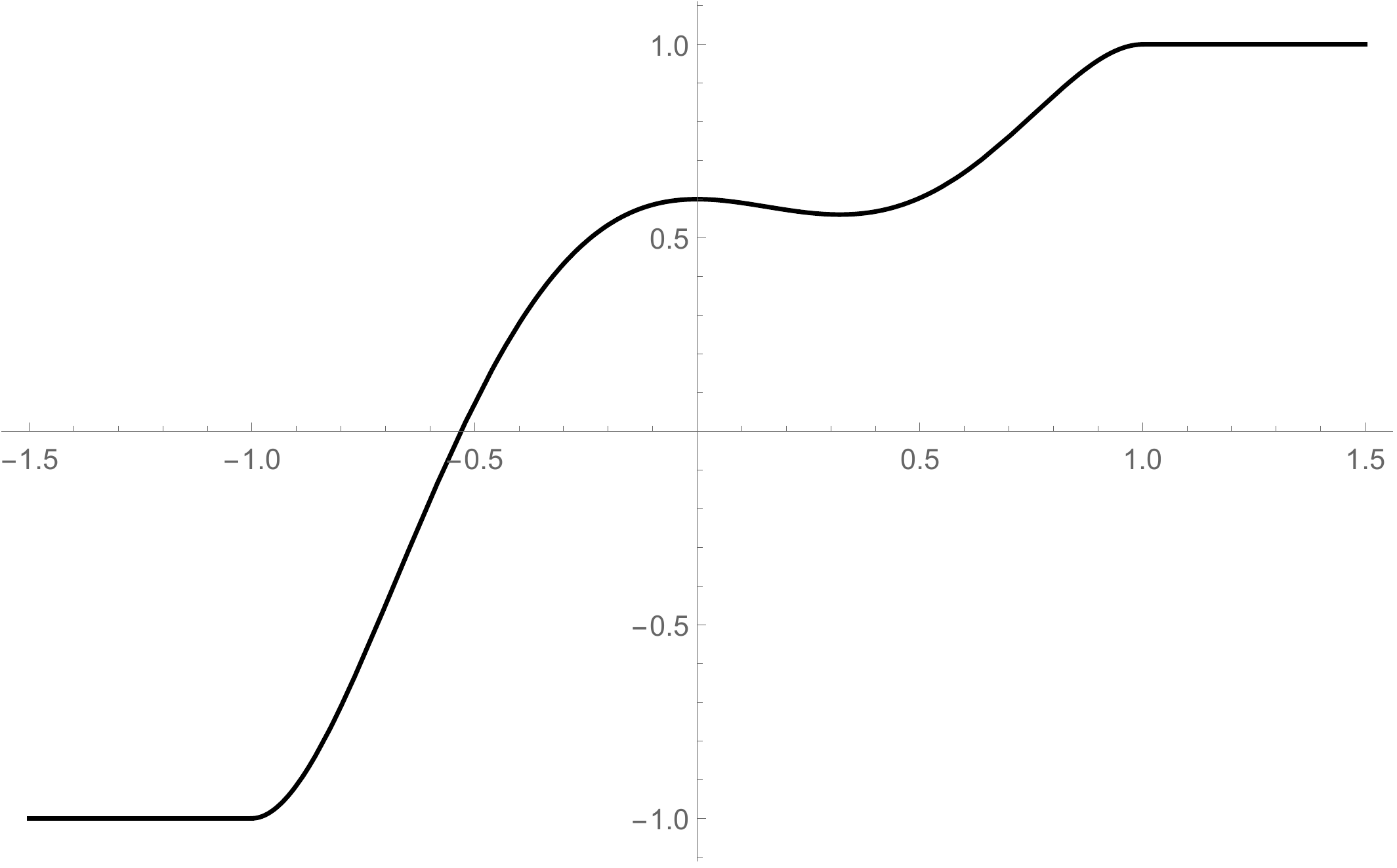}\hspace{0.3cm}\includegraphics[width=0.315\textwidth]{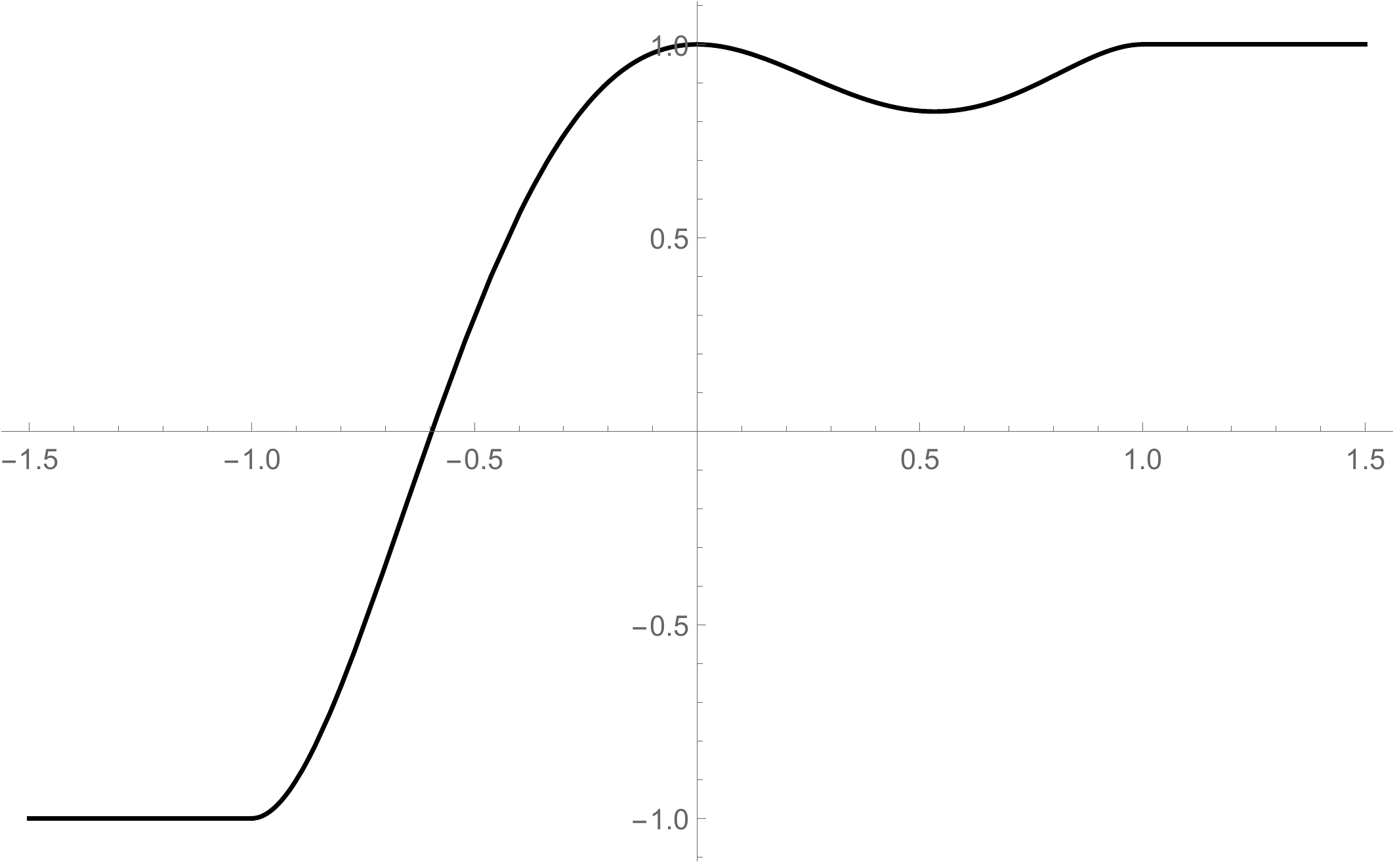}\hspace{0.3cm}\includegraphics[width=0.315\textwidth]{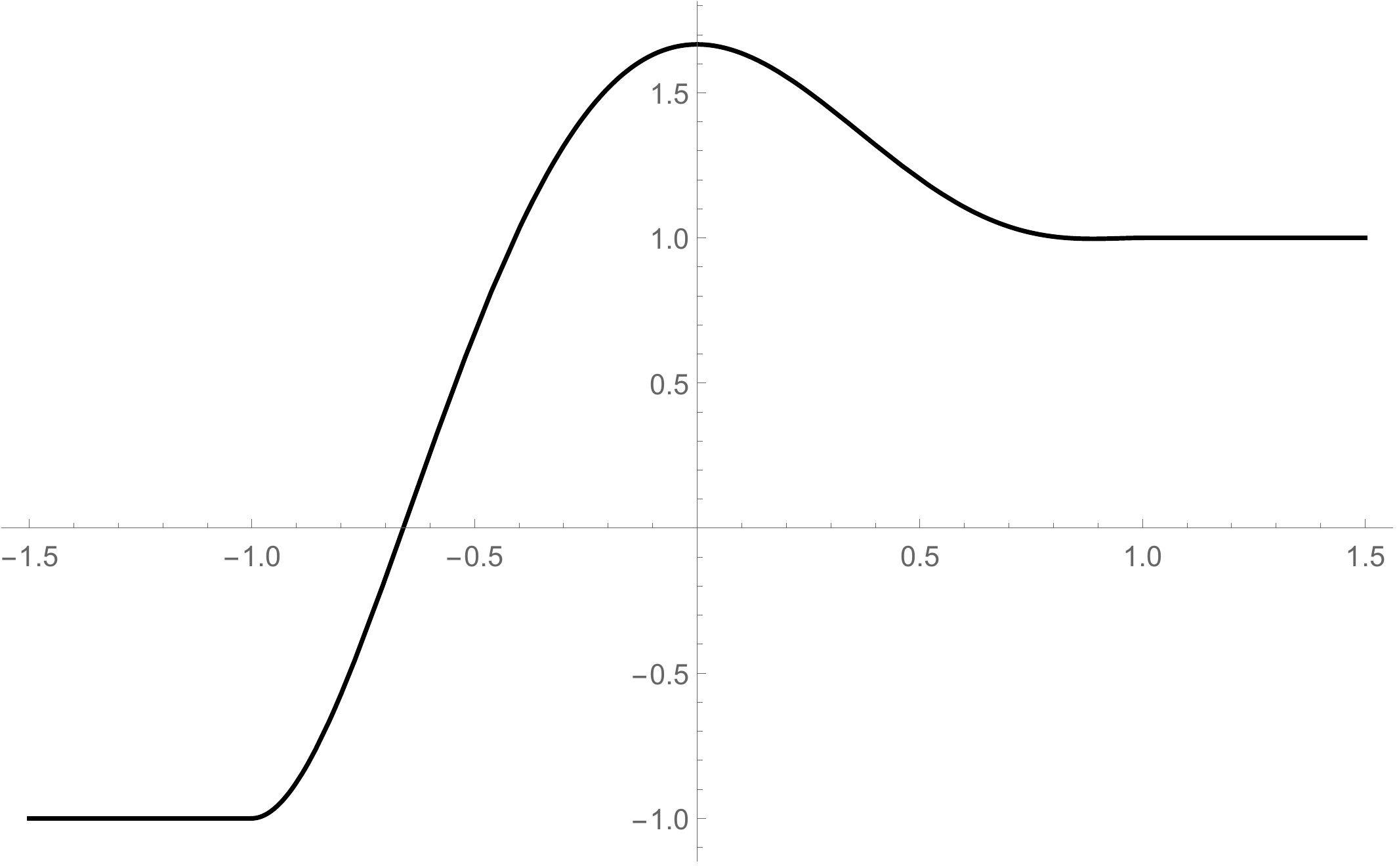}}\\
  \caption{\footnotesize{Graphic of the transition function \eqref{exe-non-monotonic-trans-func-cusp-case-1-preserva-transcritica} for $\lambda < 0$ (left), $\lambda = 0$ (center) and $\lambda > 0$ (right).}}
  \label{fig-cusp-preserva-transcritica-caso-1-trans-func}
\end{figure}

After regularization and blow-up, we obtain the slow-fast system
\begin{equation}\label{exe-non-monotonic-SF-cusp-case-1-preserva-transcritica}
\left\{
             \begin{array}{rcl}
               \varepsilon\dot{x} & = & \frac{4 y^{2}-x^{2} \Big{(}\lambda (x+1)^{2} (3 x-4)-(x-1)^{2} (3 x+4)\Big{)} \Big{(}\lambda-y^{2}-1\Big{)}}{4 (\lambda-1)}; \\
               \dot{y} & = & 1;
             \end{array}
           \right.
\end{equation}
whose critical manifold is non normally hyperbolic at the points such that $x = 0$ and $x = -\frac{8 (\lambda+1)}{15 (\lambda-1)}$. In particular, the origin will be a SF-transcritical singularity. For $\lambda > 0$, observe that $\Sigma^{s}\varsubsetneq\Sigma^{s}_{r}$. See Figure \ref{fig-bif1-cusp-sliding2}.
\begin{figure}[h]
  \center{\includegraphics[width=1\textwidth]{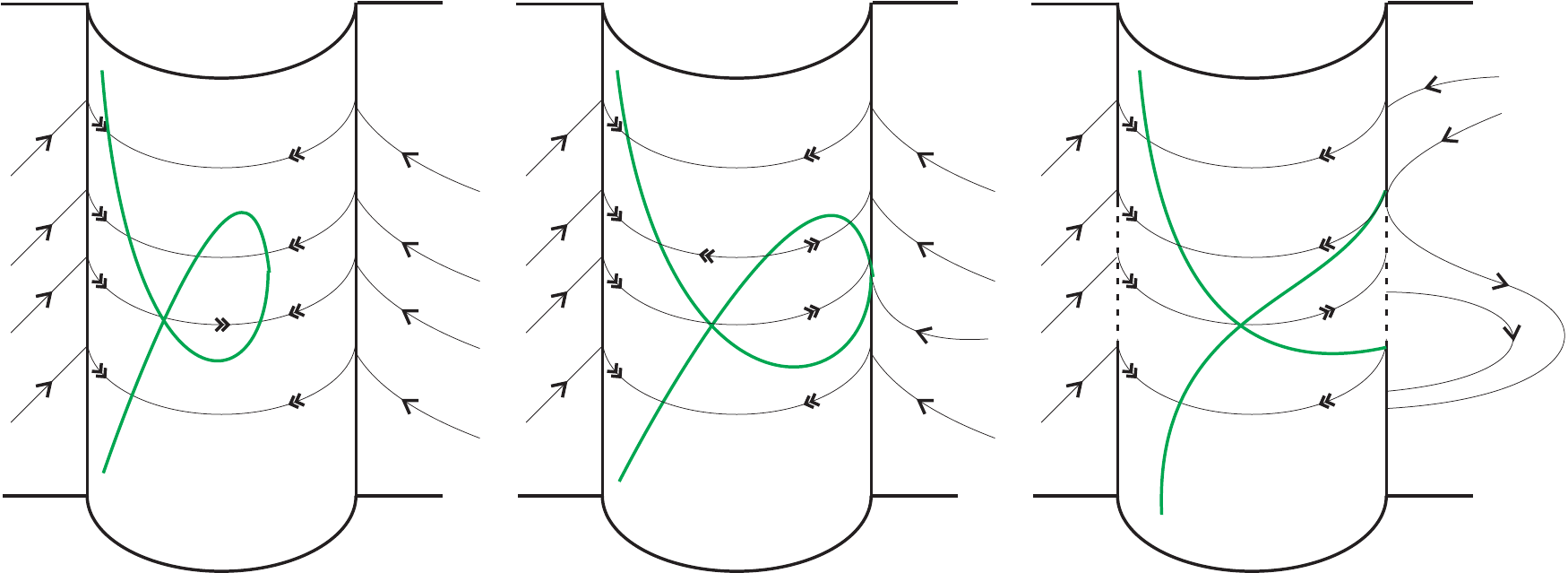}}\\
  \caption{\footnotesize{Regularized piecewise smooth vector field \eqref{eq-exe-piecewise-smooth-pwsvf-cusp-case-1} using the transition function \eqref{exe-non-monotonic-trans-func-cusp-case-1-preserva-transcritica} for $\lambda < 0$ (left), $\lambda = 0$ (center) and $\lambda > 0$ (right). The green curve is the critical set.}}
  \label{fig-bif1-cusp-sliding2}
\end{figure}
\end{example}

In general, we have the following Proposition.

\begin{proposition}\label{prop:sing_trans}
Let $Z_\lambda$, $\lambda\in(-1,1)$ be the piecewise smooth vector field \eqref{eq-exe-piecewise-smooth-pwsvf-cusp-case-1} and let $\varphi_\lambda$ be a non-monotonic transition function such that $\varphi_\lambda(0)=\frac{1+\lambda}{1-\lambda},$ $\varphi'_\lambda(0)=0,$ and $\varphi''_\lambda(0)<0$. Then, the regularized system associated with $Z_\lambda$ has a transcritical singularity at origin for all $\lambda$.
\end{proposition}
\begin{proof}
Consider the piecewise smooth vector field \eqref{eq-exe-piecewise-smooth-pwsvf-cusp-case-1} with $\lambda\in(-1,1)$. Then $f_1(x,y)=-y^2+\lambda$, $f_2(x,y)=1$, $g_1(x,y)=1,$ and $g_2(x,y)=1.$ Thus, we get that
\begin{itemize}
    \item $(f_1-g_1)(0,0)=\lambda-1\neq 0;$
    \item $\varphi_\lambda(0)=\frac{1+\lambda}{1-\lambda}.$
    \item $\scriptsize\left|
  \begin{array}{cc}
\frac{1}{4}\Big{(}(f_{1} - g_{1})\varphi''_\lambda(0)\Big{)} & 0 \\
    0 & \frac{1}{4}\Big{(}(1 + \varphi_\lambda(0))f_{1,yy} + (1 - \varphi_\lambda(0))g_{1,yy}\Big{)} \\
  \end{array}
\right|=\frac{\varphi''_\lambda(0)}{4}.$
\end{itemize}
Since $\varphi_\lambda'(0)=0$ and $\varphi_\lambda''(0)<0$, then Theorem \ref{teob} implies that the origin is a transcritical singularity for all $\lambda\in(-1,1).$
\end{proof}

The previous examples were concerned in a PS-cusp singularity such that $\Sigma^{s} = \Sigma\backslash\{0\}$. In what follows we study the case where $\Sigma^{w} = \Sigma\backslash\{0\}$.

\begin{example}
Consider the normal form of a PS-cusp singularity
\begin{equation}\label{eq-exe-piecewise-smooth-pwsvf-cusp-case-2}
Z(x,y) = \left\{
             \begin{array}{ccccc}
               X(x,y) & = & \Big{(}y^{2} + \lambda, 1\Big{)}, & \text{if} & x > 0; \\
               Y(x,y) & = & \Big{(}1, 1\Big{)}, & \text{if} & x < 0.
             \end{array}
           \right.
\end{equation}

For $\lambda = 0$, the origin is a PS-cusp singularity and $\Sigma^{w} = \Sigma\backslash\{0\}$. For $\lambda > 0$, $\Sigma^{w} = \Sigma$. Finally, for $\lambda <0$ the points $(0,\pm\sqrt{-\lambda})$ are PS-folds of $Z$, $\Sigma^{s} = \{-\sqrt{-\lambda} < y < \sqrt{-\lambda}\}$ and $\Sigma^{w} = \Sigma\backslash\Sigma^{s}$. See Figure \ref{fig-bif1-cusp-sewing1}.

Combining the ideas of Appendix \ref{sec-trans-func} and the conditions given by Theorem \ref{teob}, we construct a transition function $\varphi$ given by
\begin{equation}\label{exe-non-monotonic-trans-func-cusp-case-2}
\varphi(t) = \left\{
  \begin{array}{rcl}
    -1, & \text{if} &  t\leq -1; \\
      3 t^{6}-\frac{3 t^{5}}{2}-5 t^{4}+\frac{5 t^{3}}{2}+t^{2}+1, & \text{if} & -1 \leq t \leq 1; \\
    1, & \text{if} & t\geq 1;
  \end{array}
\right.
\end{equation}
whose derivative is zero for $t_{0} = 0$ and $t_{1,2} = \frac{1}{24} (5\pm\sqrt{89})$. See Figure \ref{fig-cusp-nao-preserva-transcritica-caso-2-trans-func}.

\begin{figure}[h]
  \center{\includegraphics[width=0.48\textwidth]{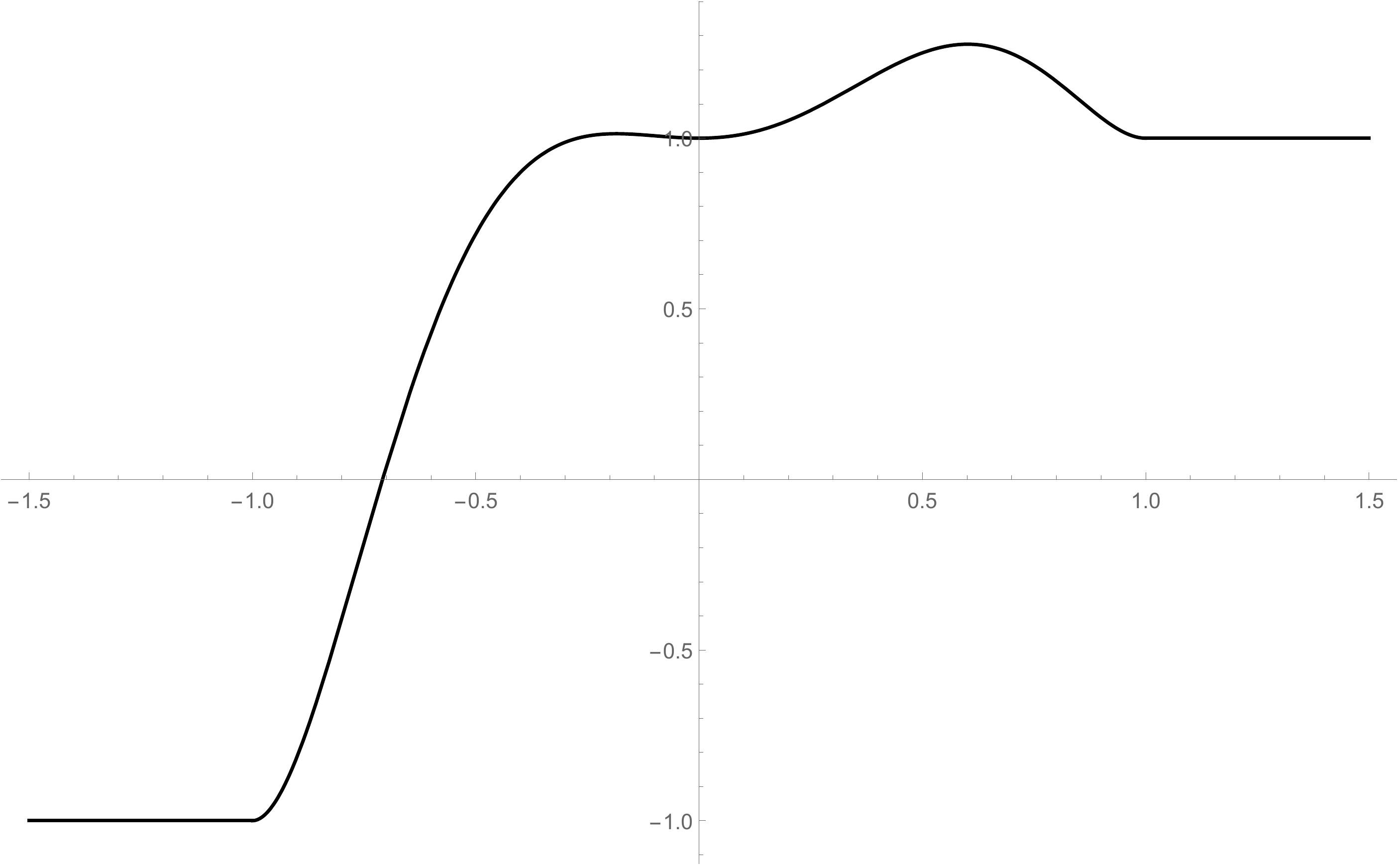}}\\
  \caption{\footnotesize{Non monotonic transition function \eqref{exe-non-monotonic-trans-func-cusp-case-2}.}}
  \label{fig-cusp-nao-preserva-transcritica-caso-2-trans-func}
\end{figure}

After regularization and blow-up, one obtains the slow-fast system
\begin{equation}\label{eq-exe-piecewise-smooth-cusp-case-2-slow-fast}
\left\{
             \begin{array}{rcl}
               \varepsilon\dot{x} & = & \frac{1}{4} \Bigg{(}\lambda (x+1)^{2} \Big{(}x (x (3 x (2 x-5)+14)-8)+4\Big{)}\\
               & &+(x-1)^{2} \Big{(}6 x^{2}+9 x+2\Big{)} x^{2} \Big{(}y^{2}-1\Big{)}+4 y^2\Bigg{)}; \\
               \dot{y} & = & 1.
             \end{array}
           \right.
\end{equation}

Observe that for $x = 0$ and $x = \frac{1}{24} (5\pm\sqrt{89})$, the critical manifold presents non normally hyperbolic points. In particular, the origin is a transcritical singularity that is destroyed for $\lambda\neq 0$. Moreover, in this example, the $\varphi$-sliding region is not empty. See Figure \ref{fig-bif1-cusp-sewing1}.

\begin{figure}[h]
  \center{\includegraphics[width=1\textwidth]{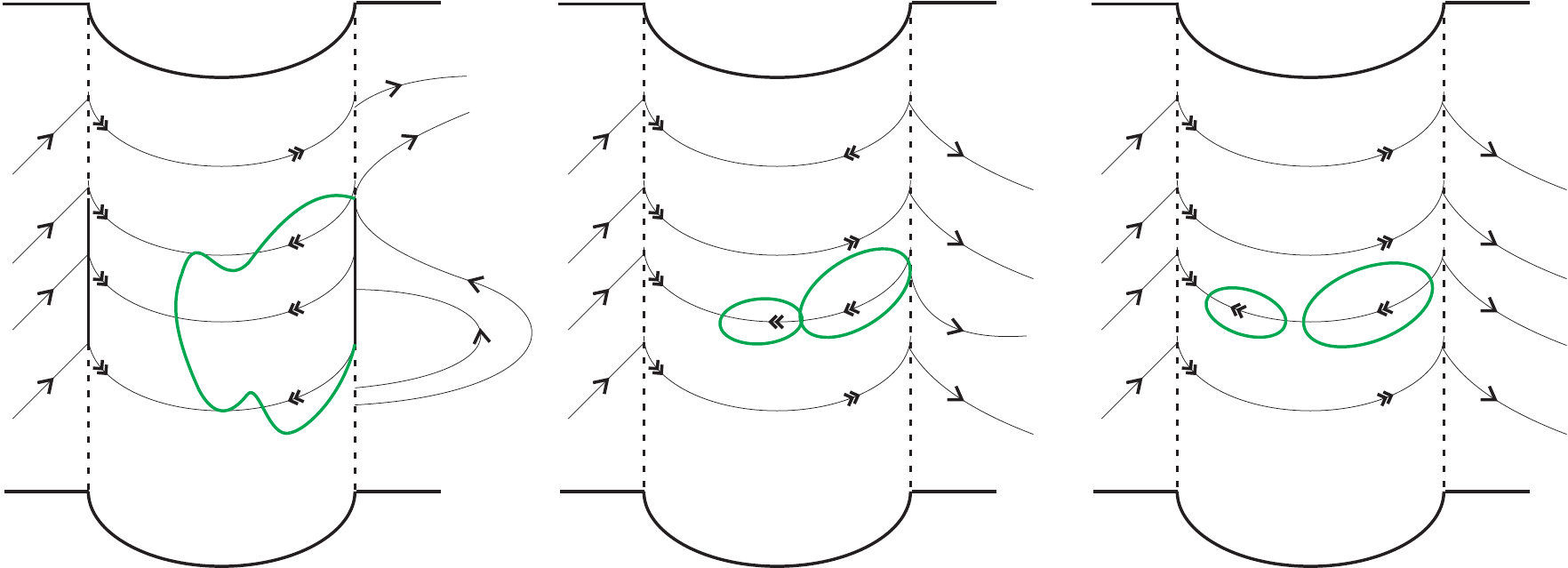}}\\
  \caption{\footnotesize{Regularization of the piecewise smooth vector field \eqref{eq-exe-piecewise-smooth-pwsvf-cusp-case-2} using the transition function \eqref{exe-non-monotonic-trans-func-cusp-case-2} for $\lambda < 0$ (left), $\lambda = 0$ (center) and $\lambda > 0$ (right). The green curve is the critical set.}}
  \label{fig-bif1-cusp-sewing1}
\end{figure}
\end{example}

In the last example, we have destroyed the SF-transcritical singularity by perturbing the parameter $\lambda$. Nevertheless, such singularity persists when the transition function is perturbed.

\begin{example}
Consider once again the piecewise smooth vector field \eqref{eq-exe-piecewise-smooth-pwsvf-cusp-case-2}, whose bifurcation diagram is given in the Figure \ref{fig-bif1-cusp-sewing1}. In this example we adopt the transition function
\begin{equation}\label{exe-non-monotonic-trans-func-cusp-case-2-preserva-transcritica}
\varphi_\lambda(t) = \left\{
  \begin{array}{rcl}
    -1, & \text{if} &  t\leq -1; \\
      -\frac{(\lambda+3) t^{6}}{\lambda-1}-\frac{3 t^{5}}{2}+\frac{(\lambda+5) t^{4}}{\lambda-1}+\frac{5 t^{3}}{2}+t^{2}+\frac{\lambda+1}{1-\lambda}, & \text{if} & -1 \leq t \leq 1; \\
    1, & \text{if} & t\geq 1;
  \end{array}
\right.
\end{equation}
which is a perturbation of the transition function \eqref{exe-non-monotonic-trans-func-cusp-case-2} and its bifurcation diagram for small $\lambda$ can be seen in Figure \ref{fig-cusp-preserva-transcritica-caso-2-trans-func}.

For $t = 0$ and $t = \frac{15 - 15\lambda\pm \sqrt{3} \sqrt{11 \lambda^{2}-278 \lambda+267}}{24 (\lambda+3)}$, the derivative of $\varphi_\lambda$ is zero. 

\begin{figure}[h]
  \center{\includegraphics[width=0.315\textwidth]{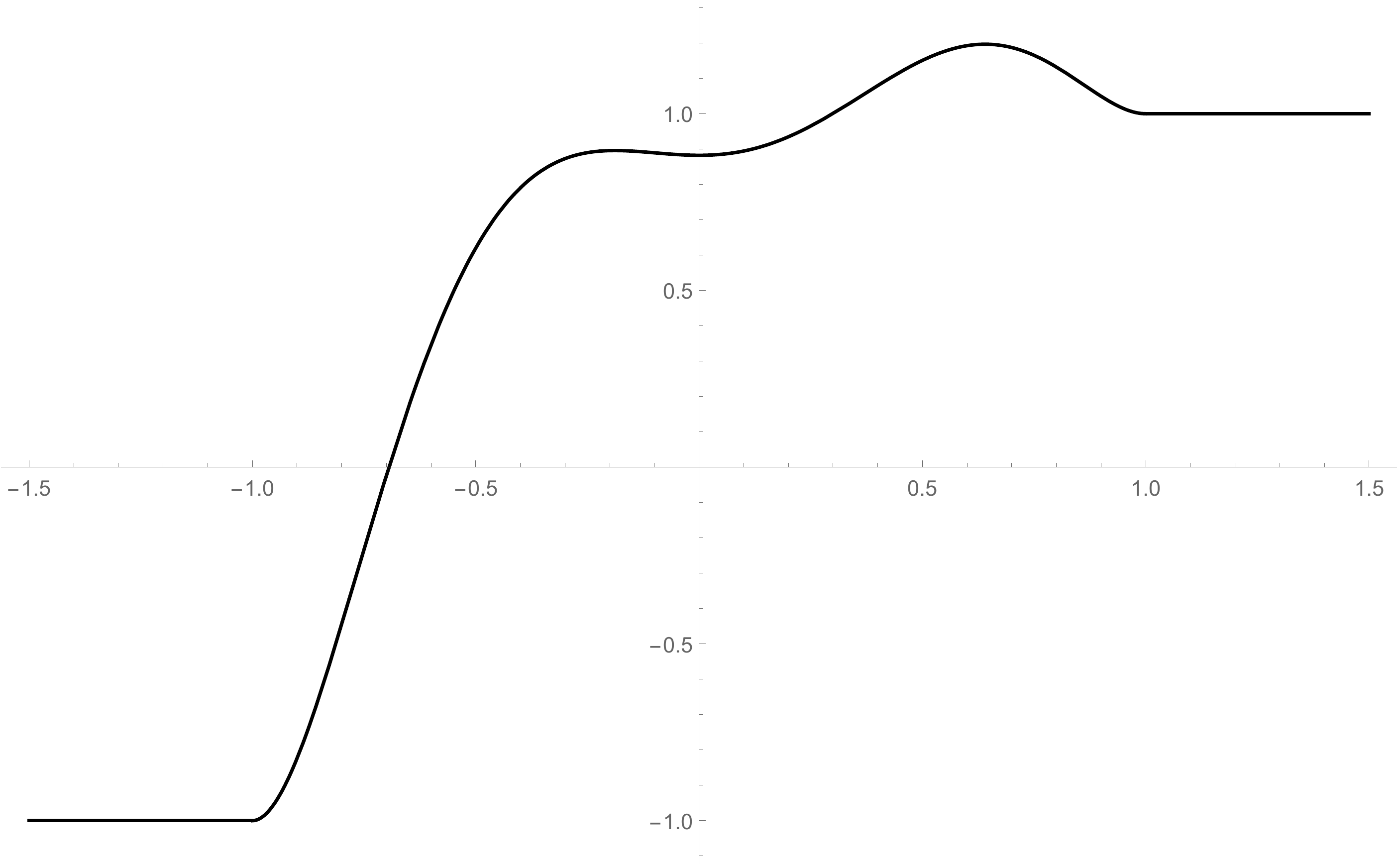}\hspace{0.3cm}\includegraphics[width=0.315\textwidth]{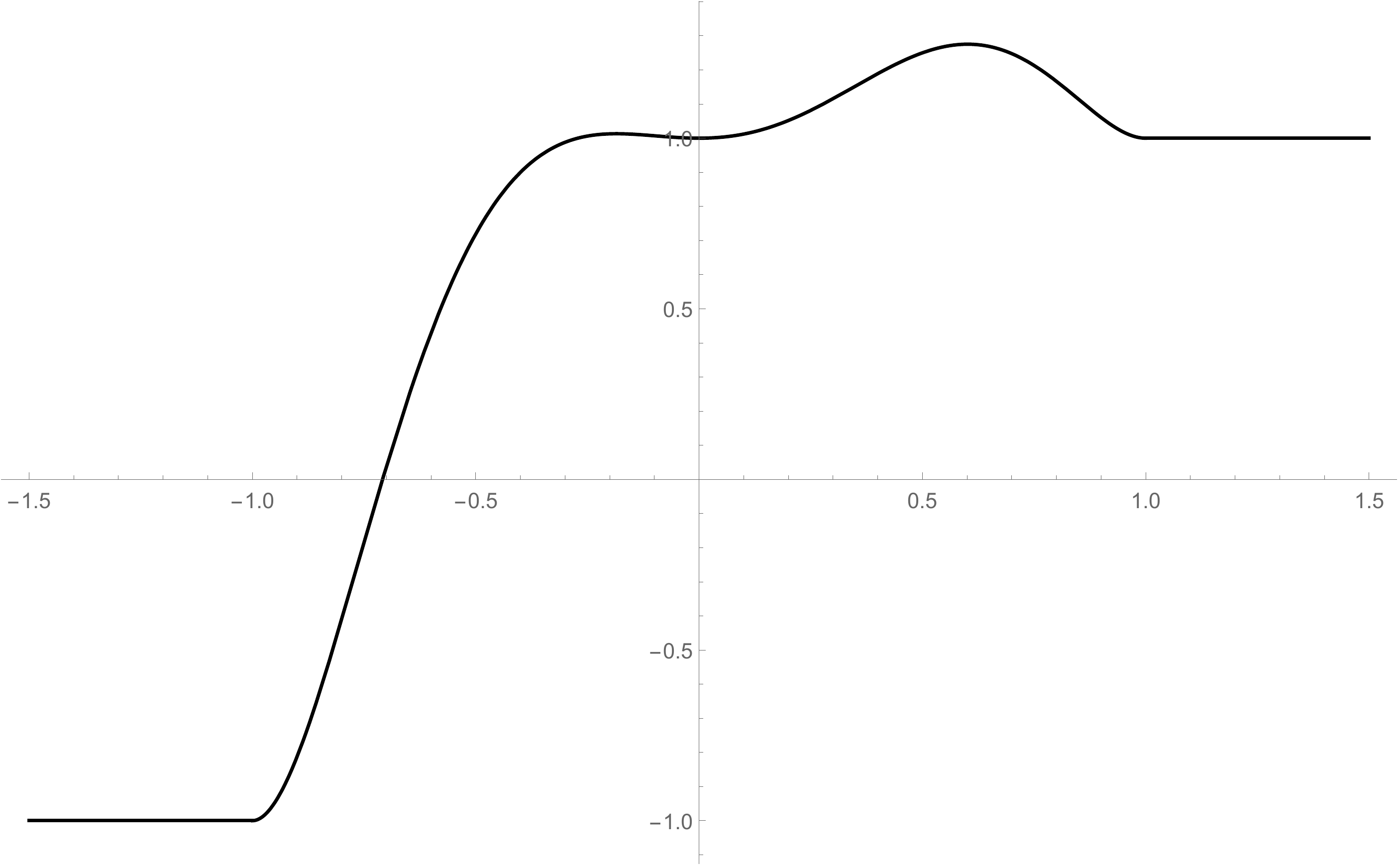}\hspace{0.3cm}\includegraphics[width=0.315\textwidth]{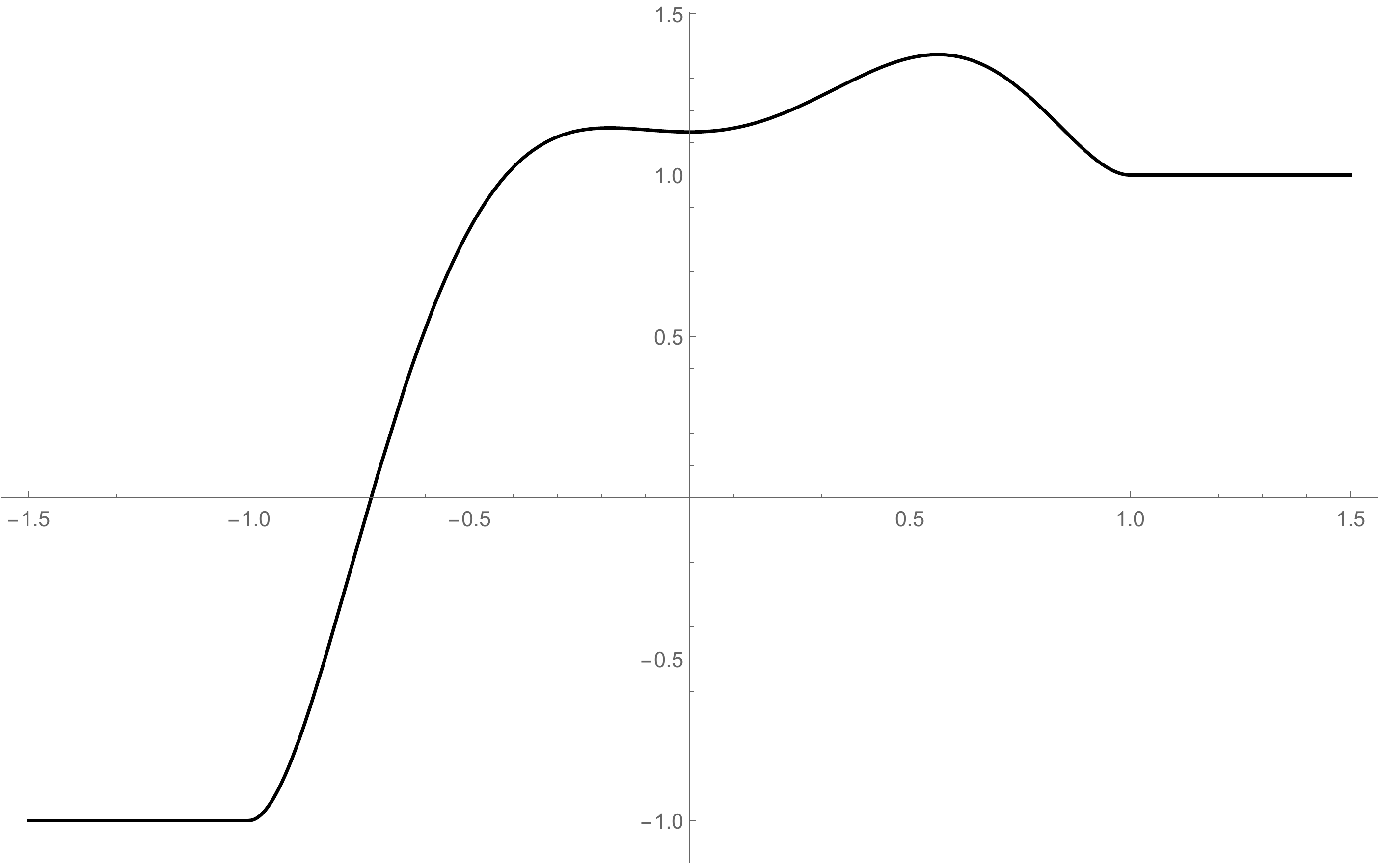}}\\
  \caption{\footnotesize{Graphics of the transition function \eqref{exe-non-monotonic-trans-func-cusp-case-2-preserva-transcritica} for $\lambda < 0$ (left), $\lambda = 0$ (center) and $\lambda > 0$ (right).}}
  \label{fig-cusp-preserva-transcritica-caso-2-trans-func}
\end{figure}

After regularization and blow-up, we obtain the slow-fast system
\begin{equation}\label{exe-non-monotonic-SF-cusp-case-2-preserva-transcritica}
\left\{
             \begin{array}{rcl}
               \varepsilon\dot{x} & = & -\frac{x^{2} \Bigg{(}\lambda (x+1)^{2} (x (2 x-1)-2)+\Big{(}6 x^{2}+9 x+2\Big{)} (x-1)^{2}\Bigg{)} \Big{(}\lambda+y^{2}-1\Big{)}+4 y^{2}}{4 (\lambda-1)}; \\
               \dot{y} & = & 1;
             \end{array}
           \right.
\end{equation}
whose critical manifold is non normally hyperbolic at the points such that $x = 0$ and $x = \frac{15 - 15\lambda\pm \sqrt{3} \sqrt{11 \lambda^{2}-278 \lambda+267}}{24 (\lambda+3)}$. In particular, the origin will be a SF-transcritical singularity. Once again, since there are values in the open interval $(-1,1)$ such that $\varphi_\lambda(t) > 1$, by Theorem \ref{mtheorem-A} it follows that $\Sigma^{s}\varsubsetneq \Sigma^{s}_{r}$. See Figure \ref{fig-bif1-cusp-sewing2}.
\begin{figure}[h]
  \center{\includegraphics[width=1\textwidth]{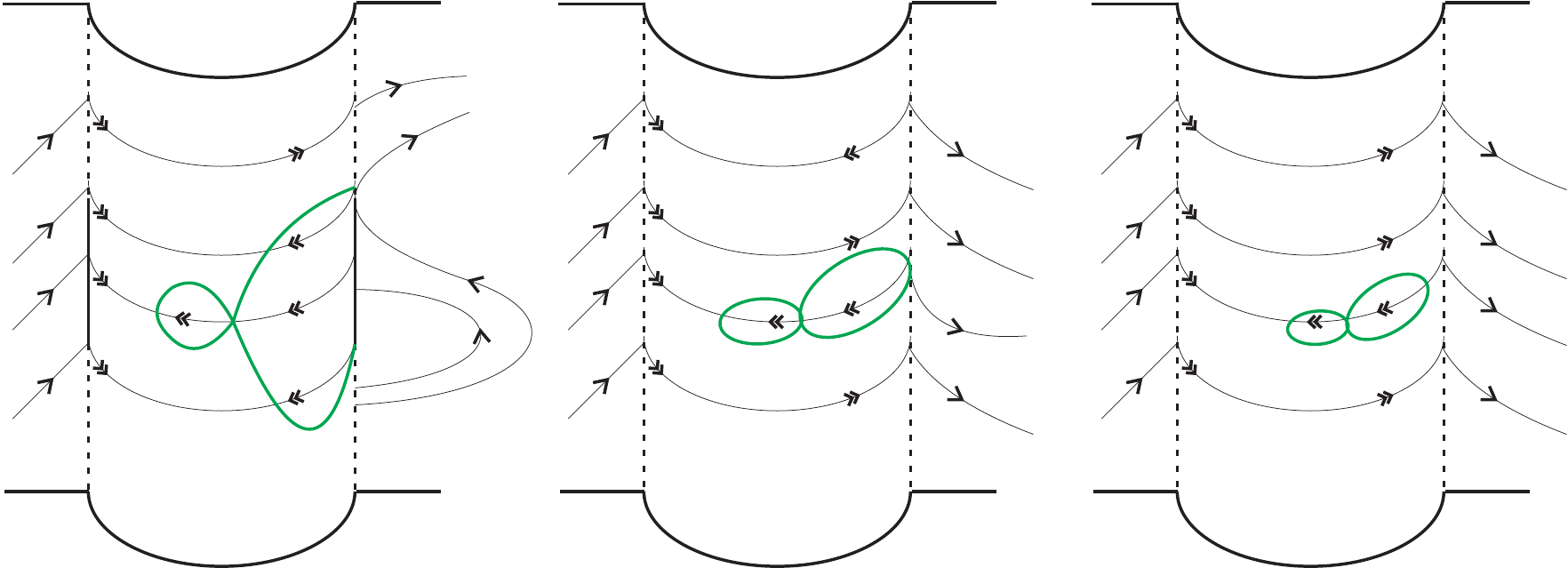}}\\
  \caption{\footnotesize{Regularization of the piecewise smooth vector field \eqref{eq-exe-piecewise-smooth-pwsvf-cusp-case-2} using the transition function \eqref{exe-non-monotonic-trans-func-cusp-case-2-preserva-transcritica} for $\lambda < 0$ (left), $\lambda = 0$ (center) and $\lambda > 0$ (right). The green curve is the critical set.}}
  \label{fig-bif1-cusp-sewing2}
\end{figure}

\end{example}

Notice that Propositions \ref{prop:sing_trans_bif} and \ref{prop:sing_trans} hold for the normal form \ref{eq-exe-piecewise-smooth-pwsvf-cusp-case-2}.


\subsection{Codimension 1 bifurcation: Visible-invisible PS-fold-fold}\label{subsec-normal-form-fold}
\noindent

This subsection is devoted to analyze a PS-fold-fold. It is important to mention that, from Theorem \ref{teob}, it is not possible to generate a SF-transcritical singularity from a PS-fold-fold.

In each PS-normal form, we see that, for $\lambda = 0$, the $\varphi$-regularization satisfies the hypotheses of Theorem \ref{mtheorem-A}, item (d). Therefore, it is not possible to extend the sliding dynamics to the origin using geometric singular perturbation theory when $\lambda = 0$.

The analysis for visible-visible, visible-invisible and invisible-invisible PS-fold-folds are completely analogous.

\begin{example}
Consider the normal form of a visible-invisible PS-fold-fold singularity
\begin{equation}\label{eq-exe-piecewise-smooth-pwsvf-VI-FOLD-case-1}
Z(x,y) = \left\{
             \begin{array}{ccccc}
               X(x,y) & = & \Big{(}2y + \lambda, 1\Big{)}, & \text{if} & x > 0; \\
               Y(x,y) & = & \Big{(}7y, 1\Big{)}, & \text{if} & x < 0.
             \end{array}
           \right.
\end{equation}

For $\lambda = 0$, the origin is a visible-invisible PS-fold-fold singularity and $\Sigma^{w} = \Sigma\backslash\{0\}$. For $\lambda < 0$, the set $\{(x,y)\in\Sigma \ ; \ 0 < y < -\frac{\lambda}{2}\}$ is a Filippov sliding region and for $\lambda > 0$, the set $\{(x,y)\in\Sigma \ ; \ -\frac{\lambda}{2} < y  < 0\}$ is a Filippov escaping region. See figure \ref{fig-bif1-fold-sewing}.

Combining the ideas of Appendix \ref{sec-trans-func} and the conditions given by Theorem \ref{teob}, we construct a transition function $\varphi$ given by
\begin{equation}\label{exe-non-monotonic-trans-func-VI-FOLD-case-1}
\varphi(t) = \left\{
  \begin{array}{rcl}
    -1, & \text{if} &  t\leq -1; \\
     -\frac{3 t^{5}}{2} -t^{4} +\frac{5 t^{3}}{2} +2t^{2} - 1, & \text{if} & -1 \leq t \leq 1; \\
    1, & \text{if} & t\geq 1;
  \end{array}
\right.
\end{equation}
whose derivative is zero for $t = 0$ and $t = -\frac{8}{15}$. See Figure \ref{fig-VI-FOLD-case-1-trans-func}.

\begin{figure}[h]
  \center{\includegraphics[width=0.5\textwidth]{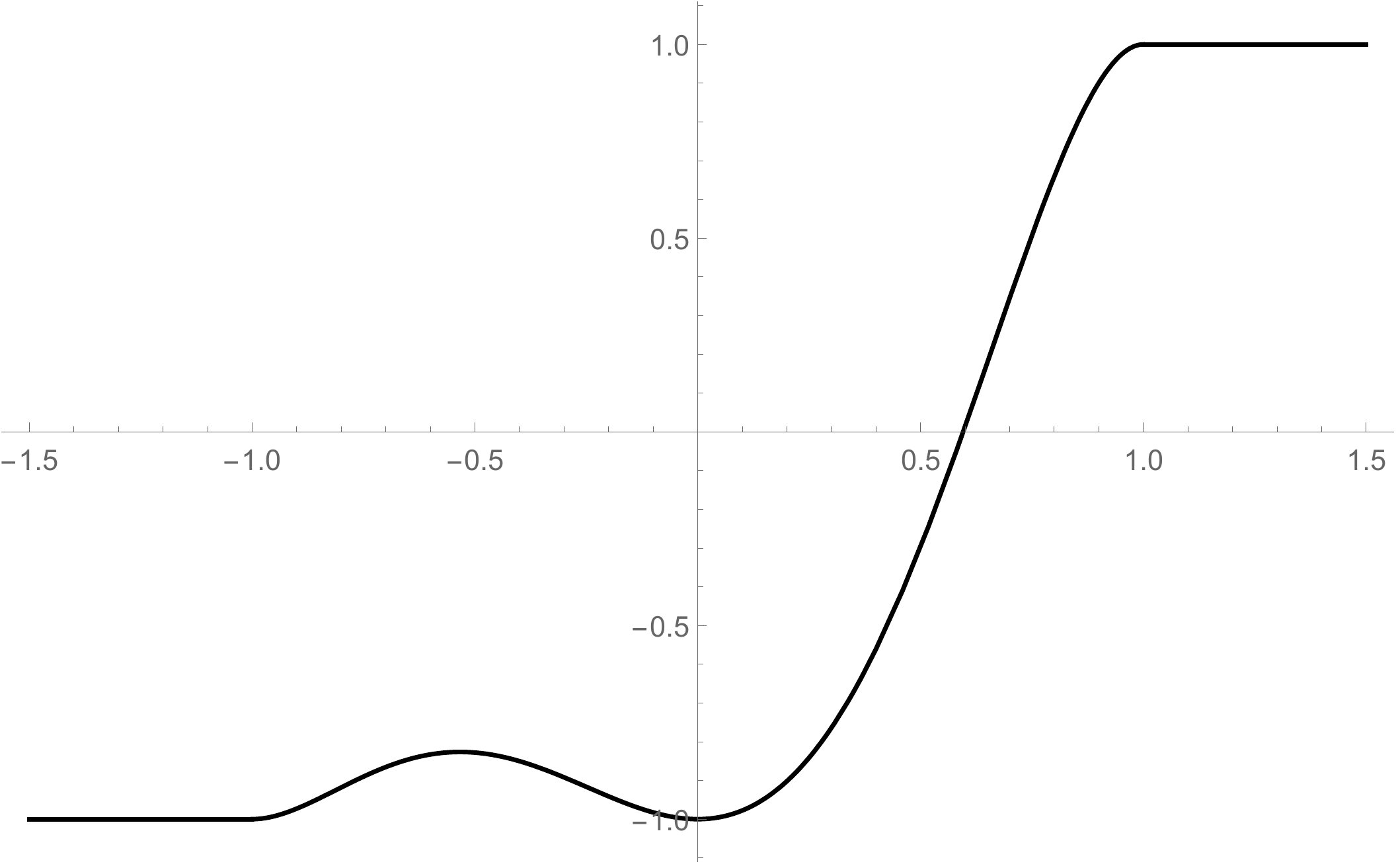}}\\
  \caption{\footnotesize{Transition function \eqref{exe-non-monotonic-trans-func-VI-FOLD-case-1}.}}
  \label{fig-VI-FOLD-case-1-trans-func}
\end{figure}

After regularization and blow-up, we obtain the slow-fast system
\begin{equation}\label{exe-non-monotonic-SF-VI-FOLD-case-1}
               \varepsilon\dot{x} = \frac{1}{4} \Big{(}28y - x^{2}(x + 1)^{2}(3x - 4)(\lambda-5 y)\Big{)}; \ \ \
               \dot{y} = 1;
\end{equation}
whose critical manifold is non normally hyperbolic at the points $x = 0$ and $x = -\frac{8}{15}$. See Figure \ref{fig-bif1-fold-sewing}

\begin{figure}[h]
  \center{\includegraphics[width=1\textwidth]{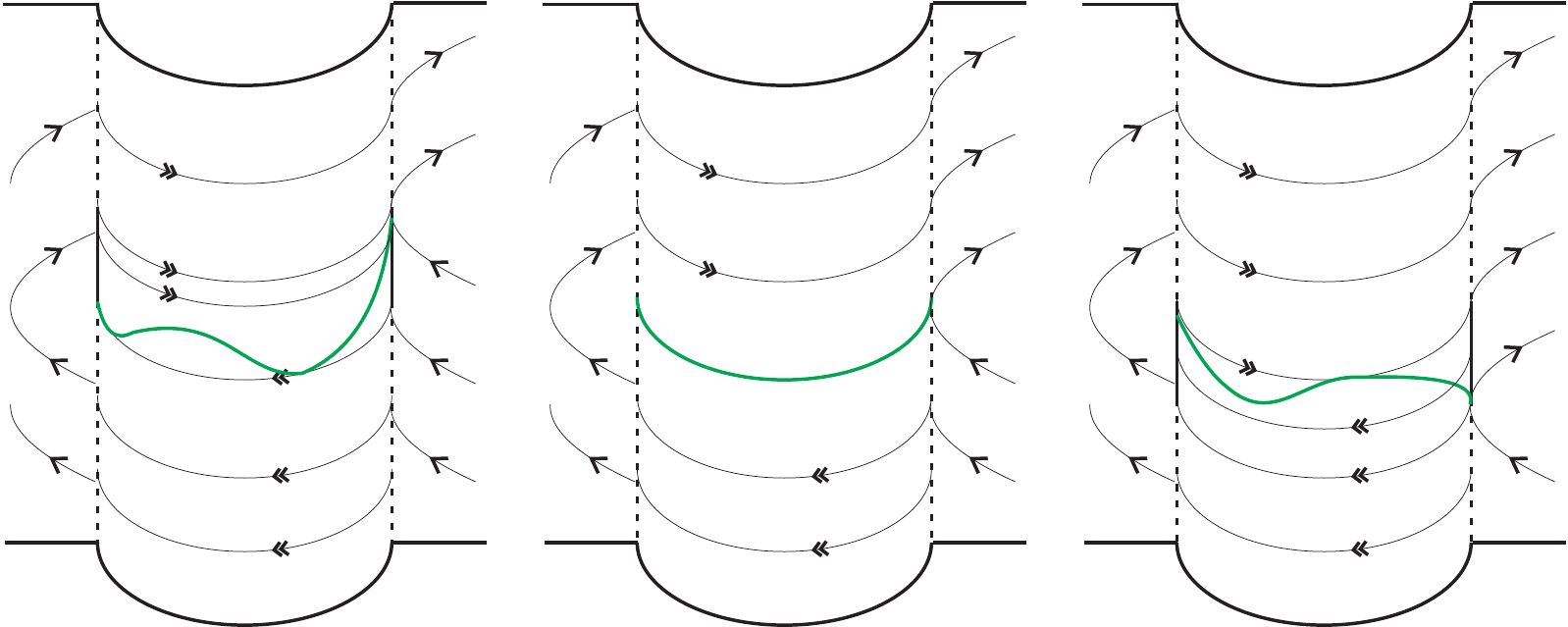}}\\
  \caption{\footnotesize{Bifurcation diagram of the regularization of the piecewise smooth vector field \eqref{eq-exe-piecewise-smooth-pwsvf-VI-FOLD-case-1} using the transition function \eqref{exe-non-monotonic-trans-func-VI-FOLD-case-1}. The green curve is the critical set.}}
  \label{fig-bif1-fold-sewing}
\end{figure}

\end{example}


\begin{example}
Consider the normal form of a visible-invisible PS-fold-fold singularity
\begin{equation}\label{eq-exe-piecewise-smooth-pwsvf-VI-FOLD-case-2}
Z(x,y) = \left\{
             \begin{array}{ccccc}
               X(x,y) & = & \Big{(}-2y - \lambda, - 1\Big{)}, & \text{if} & x > 0; \\
               Y(x,y) & = & \Big{(}7y, 1\Big{)}, & \text{if} & x < 0.
             \end{array}
           \right.
\end{equation}

For $\lambda = 0$, the origin is a visible-invisible PS-fold-fold singularity. Observe that $\Sigma^{s} = \{y > 0\}$ and $\Sigma^{e} = \{y < 0\}$. For $\lambda < 0$, we obtain $\Sigma^{w} = \{0 < y < -\frac{\lambda}{2} \}$, and for  $\lambda > 0$, we obtain $\Sigma^{w} = \{-\frac{\lambda}{2} < y < 0 \}$. See figure \ref{fig-bif1-fold-sliding}.

We adopt the transition function
\eqref{exe-non-monotonic-trans-func-VI-FOLD-case-1}. See Figure \ref{fig-VI-FOLD-case-1-trans-func}. After regularization and blow-up, we obtain the slow-fast system
\begin{equation}\label{exe-non-monotonic-SF-VI-FOLD-case-2}
               \varepsilon\dot{x} = \frac{1}{4} \Big{(}(x^{2}(3x - 4)(x+1)^{2}(9y + \lambda) + 28y\Big{)}; \ \ \ 
               \dot{y} = \varphi(x);
\end{equation}
whose critical manifold is non normally hyperbolic at the points $x = 0$ and $x = -\frac{8}{15}$. See Figure \ref{fig-bif1-fold-sliding}.

\begin{figure}[h]
  \center{\includegraphics[width=1\textwidth]{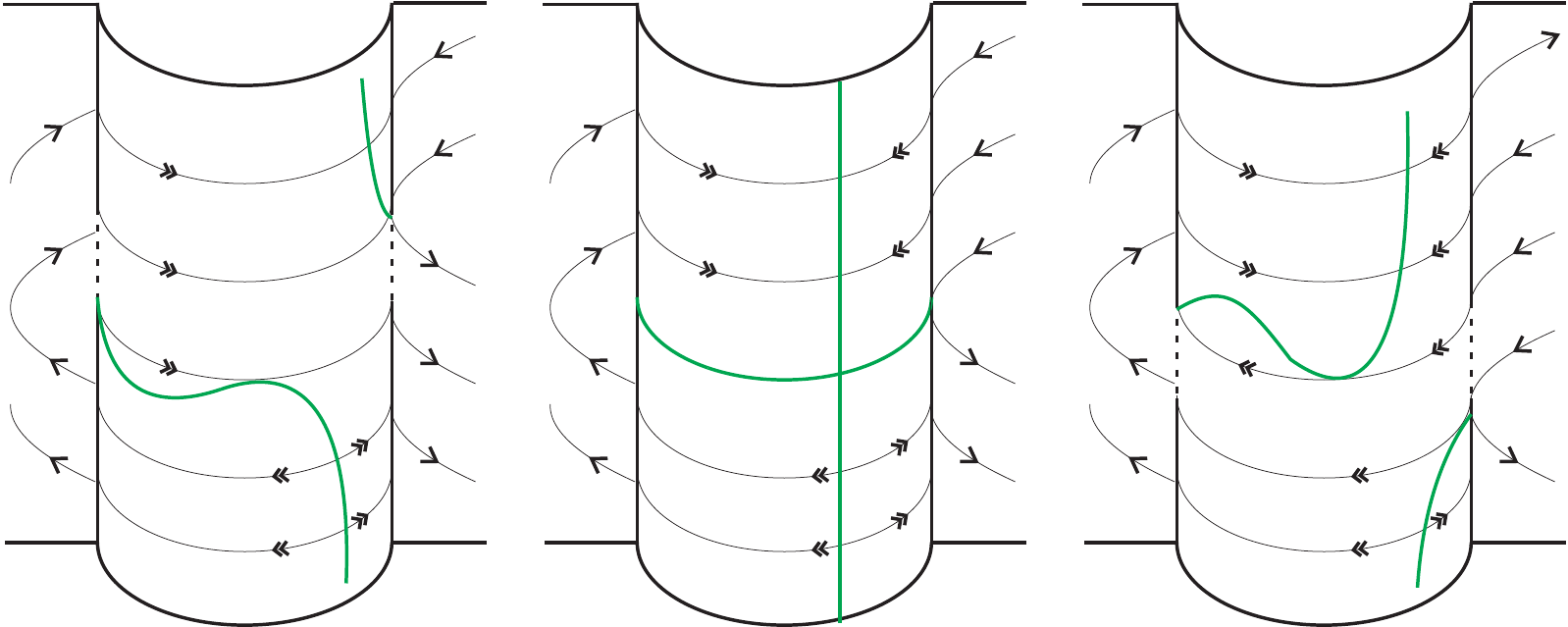}}\\
  \caption{\footnotesize{Bifurcation diagram of the piecewise smooth vector field \eqref{eq-exe-piecewise-smooth-pwsvf-VI-FOLD-case-2} using the transition function \eqref{exe-non-monotonic-trans-func-VI-FOLD-case-1}. The green curve is the critical set.}}
  \label{fig-bif1-fold-sliding}
\end{figure}

\end{example}


\subsection{Codimension 1 bifurcation: Invisible-invisible PS-fold-fold}\label{subsec-normal-form-II-fold}
\noindent

From Theorem \ref{teob}, it is not possible to generate a PS-pitchfork singularity using linear regularizations. However, Theorem \ref{teoc} assures that it is possible to generate such a SF singularity using nonlinear regularizations.

\begin{example}\label{exe-non-linear-reg}
 Let $Z = (X,Y)$ be a PSVF defined on $\mathbb{R}^2$ with $h(x,y)=x,$ $X(x,y) = ((x + 1)y + 1, -1)$, $Y(x,y) = ((x - 1)y - 1, -1)$. Consider the continuous combination of $X$ and $Y$ given by 
$$\widetilde{Z}(\lambda, x ,y)=\Big{(}(x + \lambda)y + \lambda^{3}, -1\Big{)}.$$ 

A $\varphi$-nonlinear regularization of $X$ and $Y$ is the $1$-parameter family given by $Z_{\varepsilon}(x,y) = \widetilde{Z}(\varphi(\frac{x}{\varepsilon}), x,y)$. Assume that the monotonic transition function $\varphi$ satisfies $\varphi(0) = 0$ and $\varphi'(0) \neq 0$ (for example,  $\varphi(t) = -\frac{t^{5}}{2} + \frac{t^{3}}{2} + t$, for all $t\in(-1,1)$). Thus, after nonlinear regularization and directional blow-up we obtain
\begin{equation}\label{eq-conditions-nh-non-linear}
   \varepsilon\dot{\hat{x}} = (\varepsilon \hat{x} + \varphi(\hat{x}))y + \varphi(\hat{x})^{3} =: F(\hat{x}, y, \varepsilon); \ \ \ 
   \dot{y} = -1 =: G(\hat{x}, y, \varepsilon);
\end{equation}
where $\hat{x}=\frac{x}{\varepsilon}$. Notice that \eqref{eq-conditions-nh-non-linear} satisfies conditions \eqref{eq-teoc-sf-pitchfork} and therefore the origin is a SF-pitchfork singularity. See Figure \ref{fig-pitchfork}. 

\begin{figure}[h!]
  \center{\includegraphics[width=0.48\textwidth]{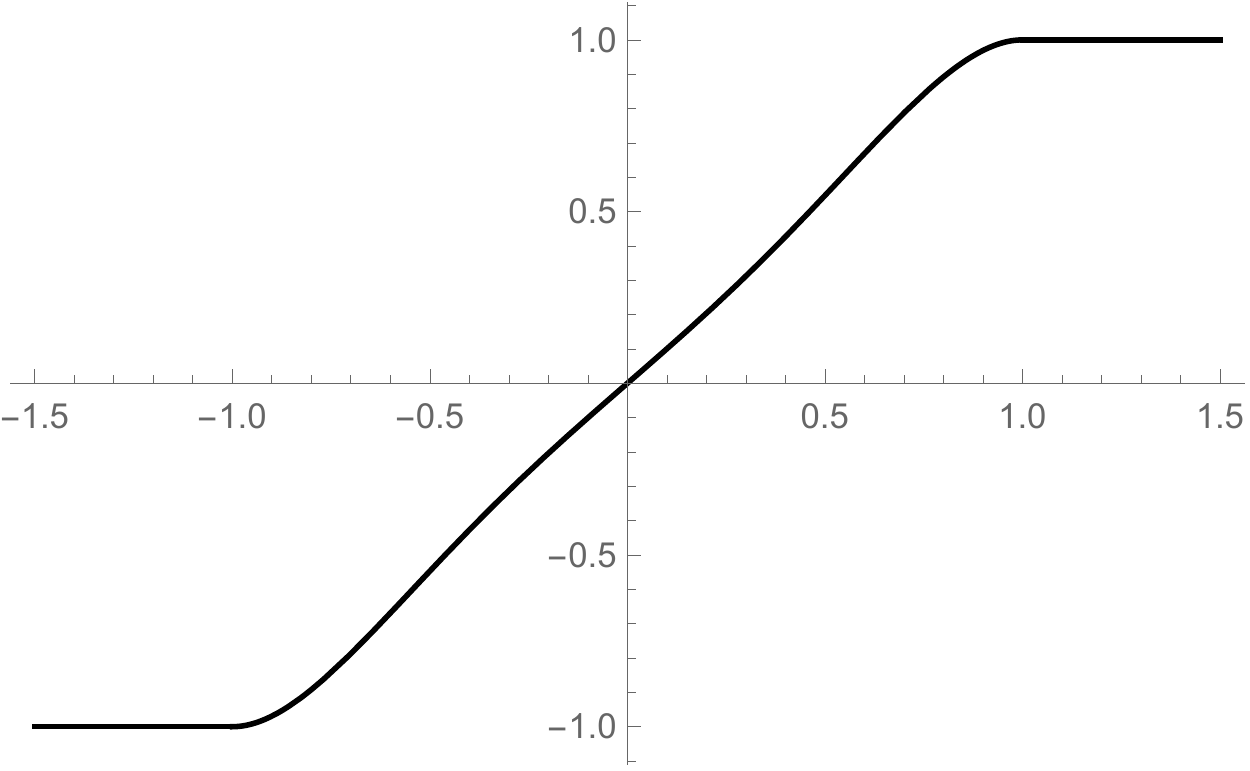}\hspace{0.7cm}\includegraphics[width=0.35\textwidth]{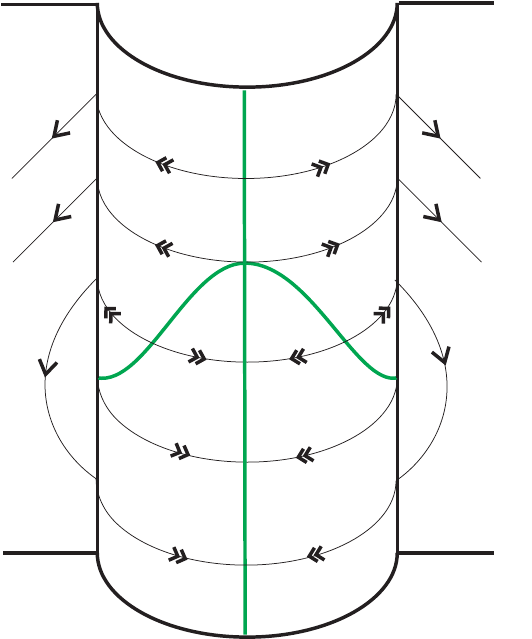}}\\
  \caption{\footnotesize{
  Graphic of the monotone transition function $\varphi$ (left) and the $\varphi$-nonlinear regularization \eqref{eq-conditions-nh-non-linear} of $f$ and $g$ (right). The critical manifold is highlighted in green.}}
  \label{fig-pitchfork}
\end{figure}

\end{example}


\section{Acknowledgements}
This article was possible thanks to the scholarship granted from the Brazilian Federal Agency for Support and Evaluation of Graduate Education (CAPES), in the scope of the Program CAPES-Print, process number 88887.310463/2018-00, International Cooperation Project number 88881.310741/2018-01.

Otavio H. Perez is supported by Sao Paulo Research Foundation (FAPESP) grant 2016/22310-0, and by Coordena\c{c}\~{a}o de Aperfei\c{c}oamento de Pessoal de N\'ivel Superior - Brasil (CAPES) - Finance Code 001. Gabriel Rondón is supported by Sao Paulo Research Foundation (FAPESP) grant 2020/06708-9. Paulo Ricardo da Silva is partially supported by São Paulo Research Foundation (FAPESP) grant 2019/10269-3. 
\appendix

\section{Constructing non-monotonic transition functions}\label{sec-trans-func}

One of our goals is to study SF-singularities related to more general regularizations, that is, regularizations given by non monotonic transition functions. In this section we discuss how to construct a suitable non monotonic transition function, and we apply these ideas in order to obtain regularizations that generate SF-singularities that do not appear in the monotonic case.

A transition function $\varphi:\mathbb{R}\rightarrow\mathbb{R}$ is sufficiently smooth and it must satisfy $\varphi(-1) = -1$ for $t \leq -1$ and $\varphi(1) = 1$ for $t \geq 1$. It remains to define $\varphi$ in the closed interval $[-1,1]$ in such a way that we obtain a sufficiently smooth function at the points $t = \pm 1$. For this purpose, one can assume that in such interval the function $\varphi$ is a \emph{polynomial} such that
\begin{equation}\label{trans-function-construc}
\varphi'(-1) = 0; \ \ \ \varphi'(1) = 0; \ \ \  \varphi(-1) = -1; \ \ \ \varphi(1) = 1. 
\end{equation}

It is clear that one can require further hypotheses, depending on the kind of critical manifold one wants to generate. Since there are (at least) four conditions that $\varphi$ must satisfy, therefore in the closed interval $[-1,1]$ the transition function is of the form
$$
\varphi(t) = a_{3}t^{3} + a_{2}t^{2} + a_{1}t + a_{0};
$$
and these four coefficients $a_{0}$, $a_{1}$, $a_{2}$, $a_{3}$ are determined solving the following system composed by four equations
$$
\left\{
  \begin{array}{rcccl}
    \varphi(-1) & = & -a_{3} + a_{2} - a_{1} + a_{0} & = & -1; \\
    \varphi(1) & = & a_{3} + a_{2} + a_{1} + a_{0} & = & 1; \\
    \varphi'(-1) & = & 3a_{3} - 2a_{2} + a_{1} & = & 0; \\
    \varphi'(1) & = & 3a_{3} + 2a_{2} + a_{1} & = & 0;
  \end{array}
\right. 
$$
which comes from conditions \eqref{trans-function-construc}. We remark that, in the closed interval $[-1,1]$, the polynomial $\varphi(t)$ may have degree greater than $3$, depending on the number of conditions that $\varphi(t)$ must satisfy in such interval.

More precisely, we have the following technical Lemma.

\begin{lemma}\label{lemma-trans-function-tech}
Suppose that the transition function must satisfy the following conditions in the interval $[-1,1]$:
\begin{equation*}
\begin{split}
\varphi'(-1) = 0; \ \ \varphi'(1) = 0; \ \ \varphi'(p_{1}) = u_{1}; \ \ \dots \ \ \varphi'(p_{k}) = u_{k}; \\
\varphi(-1) = -1; \ \ \varphi(1) = 1; \ \ \varphi(q_{1}) = v_{1}; \ \ \dots \ \ \varphi(q_{l}) = v_{l}.
\end{split}
\end{equation*}
where $p_{1},\dots,p_{k},q_{1},\dots,q_{l}\in(-1,1)$. Then the transition $\varphi$ can be considered as
$$
\varphi(t) = \left\{
  \begin{array}{rcl}
    -1, & \text{if} &  t\leq -1; \\
      a_{k+l+3}t^{k+l+3} + \dots + a_{1}t + a_{0}, & \text{if} & -1 \leq t \leq 1; \\
    1, & \text{if} & t\geq 1;
  \end{array}
\right.
$$
that is, in the interval $[-1,1]$ the transition function is given by a polynomial. Moreover, the coefficients $a_{0},\dots,a_{k+l+3}$ satisfy the following system of $k+l+4$ algebraic equations:
$$
\left\{
  \begin{array}{rcccl}
  \varphi'(-1) & = & (k+l+3)a_{k+l+3}(-1)^{k+l+2} + \dots - 2a_{2} + a_{1} & = & 0; \\
    \varphi(-1) & = & a_{k+l+3}(-1)^{k+l+3} + \dots - a_{1} + a_{0} & = & -1; \\
    \varphi'(1) & = & (k+l+3)a_{k+l+2} + \dots + 2a_{2} + a_{1} & = & 0; \\
    \varphi(1) & = & a_{k+l+3} + \dots + a_{1} + a_{0} & = & 1; \\
    \varphi'(p_{1}) & = & (k+l+3)a_{k+l+3}p_{1}^{k+l+2} + \dots + 2a_{2}p_{1} + a_{1} & = & u_{1}; \\
    \varphi(q_{1}) & = & a_{k+l+3}q_{1}^{k+l+3} + \dots + a_{1}q_{1} + a_{0} & = & v_{1}; \\
    \vdots & \vdots & \vdots & \vdots & \vdots \\
    \varphi'(p_{k}) & = & (k+l+3)a_{k+l+3}p_{k}^{k+l+2} + \dots + 2a_{2}p_{k} + a_{1} & = & u_{k}; \\
    \varphi(q_{l}) & = & a_{k+l+3}q_{l}^{k+l+3} + \dots + a_{1}q_{l} + a_{0} & = & v_{l}.
  \end{array}
\right. 
$$

\end{lemma}

Examples of this construction can be found in Section \ref{sec:normalforms}.



\begin{thebibliography}{99}

\bibitem{BuzziSilvaTeixeira} C. Buzzi, P.R. Silva, M.A. Teixeira. \emph{A singular approach to discontinuous vector fields on the plane}. \textbf{J. Diff. Eq} 231 (2006), pp 633--655.

\bibitem{CardinTeixeira} P.T. Cardin, M.A. Teixeira. \emph{Fenichel Theory for multiple time scale singular perturbation problems}. \textbf{SIAM J. Appl. Dyn. Syst.} 16(3) (2017), pp 1425--1452.

\bibitem{bernardo} di Bernardo, M., Budd, C.J., Champneys, A.R., Kowalczyk, P. \emph{Piecewise-Smooth Dynamical Systems:
Theory and Applications}. \textbf{Springer Verlag London Ltd.}, London (2008).

\bibitem{Fenichel} N. Fenichel. \emph{Geometric singular perturbation theory for ordinary differential equations}. \textbf{J. Diff. Eq.} 31 (1979), pp 53--98.

\bibitem{Filippov} A.F. Filippov. \emph{Differential Equations with Discontinuous Right-Hand Sides}. Mathematics and Its Applications (Soviet Series), Kluwer Academic Publishers, Dordrecht, 1988.

\bibitem{GuardiaSearaTeixeira} M. Guardia, T.M. Seara and M.A. Teixeira. \emph{Generic bifurcations of low codimension of planar Filippov Systems}. \textbf{J. Diff. Eq.}, 15(4) (2011), pp 1967--2023.

\bibitem{Hek} G. Hek. \emph{Geometric singular perturbation theory in biological practice}. \textbf{J. Math. Biol.} 60 (2010), pp 347--386.

\bibitem{Jones} C.K.R.T. Jones. \emph{Geometric singular perturbation theory}. In: Johnson R (ed) Dynamical systems, Montecatibi Terme, Lecture Notes in Mathematics, vol 1609. Springer, Berlin, pp 44--118, 1995.

\bibitem{Kaper} T.J. Kaper. \emph{An introduction to geometric methods and dynamical systems theory for singular perturbation problems}. In: Cronin J, O'Malley RE Jr (eds) Analyzing multiscale phenomena using singular perturbation methods. In: Proc Symposia Appl Math, vol 56. American Mathematical Society, Providence, pp 85--132, 1999.

\bibitem{KristiansenHogan} K.U. Kristiansen, S.J. Hogan. \emph{On the use of blowup to study regularizations of singularities of piecewise smooth dynamical systems in $\mathbb{R}^{3}$}. \textbf{SIAM J. Applied Dyn. Sys}, 14(1) (2015), pp 382--422.

\bibitem{KrupaSzmolyan} M. Krupa, P. Szmolyan. \emph{Extending geometric singular perturbation theory to nonhyperbolic points-fold and canard points in two dimensions}. \textbf{SIAM J. Math. Anal.}, 33(2) (2001), pp 286--314.

\bibitem{KrupaSzmolyan2} M. Krupa, P. Szmolyan. \emph{Extending slow manifolds near transcritical and pitchfork singularities}. \textbf{Nonlinearity} 14 (2001), pp. 1473--1491.

\bibitem{Kuehn} C. Kuehn. \emph{Multiple Time Scale Dynamics. Applied Mathematical Sciences}, vol 191, Springer, 2015.

\bibitem{Kuznetsov} Yu.A. Kuznetsov, S. Rinaldi, A. Gragnani. \emph{One-parameter bifurcations in planar Filippov systems} \textbf{Internat. J. Bifur. Chaos Appl. Sci. Engrg.}, 13(8) (2003), pp 2157--2188.

\bibitem{Jeffrey} M.R. Jeffrey. \emph{Hidden dynamics in models of discontinuity and switching}. \textbf{Physica D} 273--274 (2014), pp 34--45.

\bibitem{LlibreSilvaTeixeira} J. Llibre, P.R. Silva, M.A. Teixeira. \emph{Regularization of discontinuous vector fields on $\mathbb{R}^{3}$ via singular perturbation}. \textbf{J. Dynam. Diff Eq}, vol. 19, n. 2 (2006), pp 309--331.

\bibitem{LlibreSilvaTeixeira2} J. Llibre, P.R. Silva, M.A. Teixeira. \emph{Sliding vector fields via slow fast systems}. \textbf{Bull. Belg. Math. Soc. Simon Stevin}, 15 (2008), pp 851--869.

\bibitem{NovaesJeffrey} D.D. Novaes, M.R. Jeffrey. \emph{Regularization of hidden dynamics in piecewise smooth flows}. \textbf{J. Diff. Eq.}, 259 (2015), pp 4615--4633.

\bibitem{PanazzoloSilva} D. Panazzolo, P.R. da Silva. \emph{Regularization of discontinuous foliations: Blowing up and sliding conditions via Fenichel Theory}. \textbf{J Diff Eq.}, 263 (2017), pp 8362--8390.

\bibitem{SilvaSarmientoNovaes} P.R. da Silva, I.S. Meza-Sarmiento, D.D. Novaes. \emph{Nonlinear Sliding of Discontinuous Vector Fields and Singular Perturbation}. \textbf{Differ Equ Dyn Syst} (2018), DOI https://doi.org/10.1007/s12591-018-0439-1.

\bibitem{SilvaMoraes} P.R. da Silva, J.R. de Moraes. \emph{Piecewise-Smooth Slow–Fast Systems}. \textbf{J. Dyn. Control Syst}. 27 (2021), pp 67--85.

\bibitem{SotoTeixeira} J. Sotomayor, M.A. Teixeira. \emph{Regularization of discontinuous vector fields}. International Conference on Differential Equations, Lisboa, 1996, pp 207--223.

\bibitem{Wiggins} S. Wiggins. \emph{Normally Hyperbolic Invariant Manifolds in Dynamical Systems}. Springer, 1994.	

\end{thebibliography}
\end{document}